\documentclass[a4paper,10pt]{article}
\usepackage{authblk}

\usepackage[ngerman,english]{babel}
\usepackage[latin1]{inputenc}
\usepackage{csquotes}
\usepackage{amsfonts,amsmath,amsthm}
\usepackage[titletoc,title]{appendix}
\usepackage[backend=bibtex8,doi=false,eprint=false,firstinits=true,isbn=false,style=numeric-comp,url=false,maxnames=99]{biblatex}
\bibliography{bibl.bib}
\DeclareRedundantLanguages{english,german,french}{english,german,ngerman,french}

\usepackage{cases}
\usepackage{mathabx}
\usepackage{xfrac}
\usepackage{enumerate}
\usepackage{fancyhdr}
\usepackage{color}
\usepackage[colorlinks]{hyperref}
\definecolor{blue75}{rgb}{0,0,.75}
\definecolor{green75}{rgb}{0,.75,0}
\hypersetup{colorlinks=true, urlcolor=blue75,linkcolor=blue75,citecolor=green75,pdfstartview=FitB,bookmarksopen=true,bookmarksopenlevel=1}
\usepackage[a4paper, left=2.5cm, right=2.5cm, top=2.5cm,bottom=2cm]{geometry}
\usepackage{constants}
\newcommand{\parenthezises}[1]{\arabic{#1}}
\newconstantfamily{c}{
symbol=c,
format=\parenthezises,
reset={section}
}
\newconstantfamily{M}{
symbol=M,
format=\parenthezises,
reset={section}
}
\usepackage{enumitem}

\usepackage{graphicx}
\graphicspath{ {images/} }
\usepackage{wrapfig}
\usepackage{figbib}
\usepackage{subcaption}
\allowdisplaybreaks
\begin{document}
\newcommand{\R}{\mathbb{R}}
\newcommand{\N}{\mathbb{N}}
\def\diam{\operatorname{diam}}
\def\dist{\operatorname{dist}}
\def\ess{\operatorname{ess}}
\def\inner{\operatorname{int}}
\def\osc{\operatorname{osc}}
\def\sign{\operatorname{sign}}
\def\supp{\operatorname{supp}}
\newcommand{\BMO}{BMO\left(\Omega\right)}
\newcommand{\LOne}{L^{1}\left(\Omega\right)}
\newcommand{\LTwo}{L^{2}\left(\Omega\right)}
\newcommand{\Lq}{L^{q}\left(\Omega\right)}
\newcommand{\Lp}{L^{p}\left(\Omega\right)}
\newcommand{\LInf}{L^{\infty}\left(\Omega\right)}
\newcommand{\HOneO}{H^{1,0}\left(\Omega\right)}
\newcommand{\HTwoO}{H^{2,0}\left(\Omega\right)}
\newcommand{\HOne}{H^{1}\left(\Omega\right)}
\newcommand{\HTwo}{H^{2}\left(\Omega\right)}
\newcommand{\HmOne}{H^{-1}\left(\Omega\right)}
\newcommand{\HmTwo}{H^{-2}\left(\Omega\right)}

\newtheorem{Theorem}{Theorem}[section]
\newtheorem{Assumption}[Theorem]{Assumptions}
\newtheorem{borollary}[Theorem]{borollary}
\newtheorem{bonvention}[Theorem]{bonvention}
\newtheorem{Definition}[Theorem]{Definition}
\newtheorem{Example}[Theorem]{Example}
\newtheorem{Lemma}[Theorem]{Lemma}
\newtheorem{Notation}[Theorem]{Notation}
\newtheorem{Remark}[Theorem]{Remark}
\numberwithin{equation}{section}

\title{Global existence for a degenerate  haptotaxis model of cancer invasion}
\author{Anna~Zhigun, Christina~Surulescu, and Aydar Uatay}
\renewcommand\Affilfont{\itshape\small}
\affil{Technische Universität Kaiserslautern, Felix-Klein-Zentrum für Mathematik\\ Paul-Ehrlich-Str. 31, 67663 Kaiserslautern, Germany\\
  e-mail: {\{zhigun,surulescu,uatay\}@mathematik.uni-kl.de}}
\date{}
\maketitle
\begin{abstract}
We propose and study a strongly coupled PDE-ODE system with tissue-dependent degenerate diffusion and haptotaxis that can serve as a model prototype for  cancer cell invasion through the 
extracellular matrix. We prove the global existence of weak solutions and  illustrate the model behaviour by numerical simulations for a two-dimensional setting.
\end{abstract}

\section{Introduction}\label{intro}

Cancer cell migration is an essential step in the development and expansion of a tumor and its metastases. Thereby, diffusion and taxis 
are two of the main vehicles of cancer cell motility. The term 'taxis' characterizes the movement in the direction of the gradient of some signal available in the peritumoral region 
and --depending on the nature of the stimulus-- refers to chemotaxis (directed cell motion in response to a chemical concentration gradient), haptotaxis (motion follows the 
gradient of the density of tissue fibers), pH-taxis (direction of motion dictated by a pH gradient) etc. While 
 chemotaxis gradients may lack in the solution, haptotaxis seems to be indispensable, as the cells need to adhere to the ECM in order to support their motion \cite{adams02}, but also for 
 information exchange with their surroundings, the latter being closely related to survival and proliferation \cite{legate-etal,schwartz_assoian01}, see also \cite{pickup} for a comprehensive 
 review. For these reasons we focus here on diffusion and haptotaxis. The latter is triggered 
 by an unsoluble stimulus: the fibers of the extracellular matrix (ECM) and their density\footnote{and orientation, but this is not the case in the present model type}.\\[1ex]
\noindent
Biological experiments suggest that:
\begin{itemize}
  \item[(i)] 
  enhanced interactions with the surrounding tissue favor cell motility (in particular, diffusivity) \cite{friedl-wolf};
  \item[(ii)] in those areas where the cells and the ECM are tightly packed, the diffusivity and the advective effects of haptotaxis (and hence also the invasion into the tissue) 
  are limited \cite{lu};
 \item[(iii)] no cell migration (in particular no diffusivity) occurs in regions where the tissue is absent (see above);
  \item[(iv)] cells  propagate through the ECM with a finite speed.
\end{itemize}
 \noindent
 As mentioned already, haptotaxis is connected to directioning the motion along the gradient of an immovable stimulus (density of tissue fibers). Therefore, the evolution of the 
 latter is characterized by way of an ODE. Since it contains no spatial diffusion, this ODE corresponds to an everywhere degenerate reaction-diffusion PDE and has 
 no regularising effect. When strongly coupled to a PDE for the cell density via a haptotactic transport term, this causes a considerable difficulty for the analysis. \\ 

\noindent
Previous models for cell migration involving haptotaxis and operating on the macroscopic scale of population densities have been proposed\footnote{we only consider here 
pure continuum models and omit both discrete and hybrid settings} e.g., in \cite{anderson-etal_2000,chaplain-anderson-03,JMGL06} upon relying on equilibrium of fluxes (diffusion and 
haptotaxis, possibly with some other kinds of taxis as well). The mathematical analysis of this model class was most often performed in the case with linear diffusion for 
the tumor cell density, see e.g., \cite{marciniak-ptashnyk,tao-09,walker-webb-07} and only recently approached for settings allowing for nonlinear diffusion \cite{tao-winkler-11,
wang16}. Still, in the pure macroscopic framework, nonlocal models including cell-cell and cell-tissue interactions within a sensing radius by way of integral terms 
have been proposed and simulated \cite{armstrong-sherratt-06,gerisch-chaplain-08,painter-etal-10,sherratt-etal-09}. The mathematical analysis (well posedness of 
classical solutions) of a couple of models in that class, 
however in some simplified settings --yet with linear diffusion-- has been done in \cite{clsw-11,smlc-09}. \\

\noindent
Multiscale models for cancer cell migration involving haptotaxis and 
coupling subcellular dynamics (microlevel) with population dynamics (macroscale) have been recently proposed and investigated with respect to well posedness in \cite{mss-15,SSW}, 
also considering nonlinear diffusion for the tumor cell density. A further micro-macro model for acid-mediated tumor invasion through tissue and allowing for gap formation 
at the tumor interface was proposed in \cite{hiremath-surulescu-15} and its global well-posedness shown; the model also accounts for stochastic effects, nonlinear diffusion, and 
repellent taxis. Yet another multiscale model class for tumor invasion with (chemo- and) haptotaxis is that considered and analyzed in \cite{kelkel-surulescu,lorenz-surulescu}. 
Those models couple subcellular dynamics (ODEs) with mesoscopic kinetic transport equations describing individual cell behavior and the evolution of tissue fiber density, 
and with the macrolevel dynamics of a chemoattractant concentration. Those models are able to account for nonlinear diffusion, meaning that the diffusion coefficient 
in the equation for cell density is allowed to depend on the solution itself. \\

\noindent
In the quasilinear system handled in \cite{SSW} the diffusion coefficient of the cell density 
depends, moreover, upon the local interaction between the tumor cells and the ECM fibers. However, that coefficient was still assumed to be nondegenerate (at least as long as the solutions remain bounded), and thus the model did 
not capture the features (iii) and (iv) listed 
above. In order to account for all properties (i)-(iv) we develop in this paper a degenerate-diffusion system, thereby keeping only two components: the density of the 
tumor cells and the density of the tissue fibers, hence studying a model with diffusion and haptotaxis only. As we focus here on the degeneracy issue in the framework of a 
haptotaxis model we ignore the multiscality and stay on the population level. For this new prototype model we prove the global existence of weak solutions. The uniqueness and 
boundedness of solutions remain open.\\

\noindent
This paper is organised as follows: In {\it Section  \ref{model}} we set up the mathematical model, followed by fixing some notations in {\it Section \ref{not}} and by the statement of the 
problem and the main result in {\it Section \ref{problem}}. The subsequent {\it Sections \ref{aproxi}}, {\it \ref{apriori}}, and {\it \ref{existence}} are dedicated to the proof of this result, by 
constructing sequences of nondegenerate approximations to the original problems, proving some apriori estimates for these approximations, and passing to the limits in the 
approximations, 
respectively. 
%
%
In {\it Section~\ref{numerics}}, we illustrate the possible model behaviour by performing numerical 
simulations in 
the two-dimensional case and compare the  results with those reproduced for a previous model with nondegenerate diffusion of the tumor cells. 
Finally, we provide in {\it Section \ref{discuss}} a short discussion of the obtained results and set them in context with respect to other models with degenerate diffusion. 
The paper includes an {\it Appendix} with two auxiliary results ({\it Lemmata~\ref{LemA1}} 
and {\it~\ref{LemA2}}) dealing with weak and almost everywhere convergence and being of independent interest. 
\section{The model}\label{model}
In this section we introduce an ODE-PDE system for two variables: the cancer cells  density 
$c$ and the density $v$ of ECM tissue fibers, both depending on time and position on a smooth bounded domain $\Omega \subset \R^N$.
Our system, a variant of the one introduced in \cite{SSW}, has the following form:
\begin{subequations}\label{hapto}
\begin{alignat}{3}
 &\partial_t c=\nabla\cdot\left(\frac{\kappa_c vc}{1+vc}\nabla c-\frac{\kappa_v c}{(1+v)^2}\nabla v\right)+\mu_cc(1-c-\eta v)&&\text{ in }\R^+\times\Omega,\label{c}\\
 &\partial_t v=\mu_vv(1-v)-\lambda vc&&\text{ in }\R^+\times\Omega,\label{v}\\
 &\frac{\kappa_c vc}{1+vc}\partial_{\nu} c-\frac{\kappa_v c}{(1+v)^2}\partial_{\nu} v=0&&\text{ in }\R^+\times\partial\Omega,\label{bc}\\
 &c(0)=c_0,\ v(0)=v_0 &&\text{ in }\Omega,
\end{alignat}
\end{subequations}
where $\kappa_c, \kappa_v, \mu_c, \eta, \mu_v, \lambda$ are some positive constants. System \eqref{hapto} consists of a degenerate parabolic PDE which describes 
the evolution of the tumor cell density and an ODE for the evolution of the tissue density, supplemented by the initial and the 'no-flux' boundary conditions. The latter 
is realistic, since the cancer cells do not leave the tissue hosting the original tumor. \\
\noindent
Equation \eqref{c} for the tumor cell density includes two nonlinear spatial
movement effects: degenerate diffusion and  haptotaxis transport. \\
\noindent	      
The nonlinear diffusion coefficient in \eqref{c} is taken to be of the form $$D_c(v,c):=\frac{\kappa_c vc}{1+vc},$$ where the positive constant $\kappa_c$ accounts for the 
adhesivity between the tumor cells and the fiber. Notice that the source of degeneracy is twofold: the diffusion coefficient can become zero when $c=0$, but also 
when $v=0$. Our choice of the diffusion coefficient is less restrictive than previous settings which involve some powers of the solution, chosen in such a way as to render the 
mathematical analysis more amenable. Instead, the choice of our degenerate diffusion coefficient is motivated by the biological phenomenon under consideration, more precisely by 
the four properties of the cell spreading which we proposed in 
 {\it Section \ref{intro}}. 
Indeed, let us consider the product $cv$ as a measure of interaction between the cells and the tissue. Then, we have that:
\begin{itemize}
 \item $D_c$ is monotonically increasing in $vc$;
 \item $\underset{vc\rightarrow\infty}{\lim}D_c(v,c)=\kappa_c<\infty$; 
 \item $D_c(0,c)=0$ for all $c\in\R_0^+$;
 \item $D_c(v,c)\underset{c\rightarrow0}{\approx}
\kappa_c vc$ for all $v\in\R_0^+$. 
\end{itemize}
The first three properties above clearly correspond to (i)-(iii) from the introduction. As for the last property, the porous-medium type degeneracy with respect to 
variable $c$ is known to ensure a finite speed of propagation, which provides the condition (iv) in {\it Section \ref{intro}}. It seems that \cite{Kawasaki1997} was the 
first paper involving a diffusion coefficient of the form $\kappa_c vc$, there in the context of bacterial biofilm dispersal.\\
\noindent
The signal-dependent haptotactic sensitivity function 
\begin{equation}\label{hapto-sensit}
 \chi(v):=\frac{\kappa_v}{(1+v)^2}
\end{equation}
is obtained upon accounting for receptor binding to ECM fibers. Here, however, we avoid including specific subcellular dynamics and simplify the setting by looking directly 
at cell-tissue interactions instead of receptor-ligand bindings. Indeed, consider the 'mass action kinetics'\footnote{for simplicity, on this level we only take into 
account conservative interactions (no decay, no proliferation); here $[cv]$ denotes the amount of cells bound to the tissue}
\begin{equation*}
 c+v\underset{k^-} {\stackrel{k^+}\rightleftharpoons }[cv],
\end{equation*}
leading to the ODE system 
\begin{align*}
\partial _tc&=k^-[cv]-k^+cv \\
\partial _tv&=k^-[cv]-k^+cv \\
\partial _t[cv]&=k^+cv-k^-[cv].
\end{align*}
As the binding kinetics is very fast, we may assume that the corresponding steady-state is quickly achieved, hence from the last equation above we obtain 
\begin{equation}\label{rez-c}
[cv]=\frac{k^+}{k^-}cv.
\end{equation}
Furthermore, we assume the total amount of cells is conserved during this short time span, hence 
$$c+[cv]=\text{const}.$$
This leads to $c=\text{const}-[cv]$ and plugging into \eqref{rez-c} and using the notation $\kappa :=\frac{k^+}{k^-}$ we get 
$$[cv]=\frac{\text{const}\cdot v}{\kappa +v}.$$
The haptotaxis equation is obtained by equilibrium of fluxes, but it can also be deduced from a master equation written with the aid of the corresponding probabilities (rates) 
$T_i^{\pm}$ 
of transition from a position $i$  into the adjacent one $i-1$ or $i+1$, respectively. With the gradient-based choice (see e.g., \cite{othmer-stevens-97})
$$T_i^{\pm}:=\alpha +\beta (\tau (v_i)-\tau (v_{i\pm 1})),\quad \alpha >0,\ \beta \ge 0,$$
where $\tau $ is a known differentiable function\footnote{satisfying $\tau '(v)>0$ in the case with positive haptotaxis}, we get for the haptotaxis coefficient the form $\chi (v)c$, with $\chi (v)=2\beta \tau '(v)$. As $\tau (v)$ characterizes the 
(chemical) mechanism of measuring tissue densities, we can interpret $\tau (v)$ as giving the amount of cell surface receptors bound to the tissue fibers, or --even 
further simplified, to avoid introducing the subcellular scale explicitly-- the amount of cells bound to ECM fibers, hence $\tau (v)=[cv]= \frac{\text{const}\cdot v}{\kappa +v}$. This leads to 
$$\chi (v)=2\beta \frac{\text{const}\cdot \kappa }{(\kappa +v)^2},$$
which is of the form \eqref{hapto-sensit} announced above.\\
\noindent
The equation \eqref{v} for the tissue density $v$ is an ODE. It contains no spatial movement effects since the ECM fibers do not move on their own. They can be deformed, at most, 
but we ignore here such deformations.

\section{Basic notation and functional spaces}\label{not}
Partial derivatives, in both classical and distributional sense, with respect to variables  $t$ and $x_i$, will be denoted respectively by $\partial_t$ 
and $\partial_{x_i}$. Further, $\nabla$, $\nabla\cdot$ and $\Delta$ stand for the spatial gradient, divergence and Laplace operators, respectively. $\partial_{\nu}$ is 
the derivative with respect to the outward unit normal of $\partial\Omega$.

\noindent
We assume the reader to be familiar with the standard $L^p$, Sobolev, and H\"older spaces and their usual properties, as well as with the  more general $L^p$ spaces of functions with values in general Banach spaces and with anisotropic Sobolev spaces. In particular, we need the space
\begin{align}
 W^{-1,1}(\Omega):=\left\{u\in D'(\Omega)|\ u=u_0+\sum_{k=1}^N\partial_{x_i}u_i\text{ for some }u_i\in\LOne,\ i=1,\dots,N\right\}.\nonumber
\end{align}
\noindent
For $p\in[1,\infty]\backslash\left\{2\right\}$, we write $||\cdot||_p$ in place of the $||\cdot||_{\Lp}$-norm. Throughout the paper, $||\cdot||$ stands for 
the $||\cdot||_{\LTwo}$-norm and $(u,v)$ for $\int_{\Omega}u(x)v(x)\,dx$, while $\left<\cdot,\cdot\right>$ is reserved for the duality pairing between 
$W^{1,\infty}(\Omega)$ and its dual $(W^{1,\infty}(\Omega))'$. 

\noindent
We denote the Lebesgue measure of a set $A$ by $|A|$ and by $\inner A$ its interior.

\noindent
Finally, we make the following useful convention: For all indices $i$, the quantity $C_i$ denotes a non-negative constant or, alternatively, a non-negative function, 
which is non-decreasing in each of its arguments.
\section{Problem setting and main result}\label{problem}
In this section we propose a definition of weak solutions to system \eqref{hapto} and state our main result under the following assumptions:
\begin{Assumption}[Initial data]\label{ini}~
\begin{enumerate} 
 \item 
  $c_0\geq0,\ c_0\notequiv 0, \ c_0\ln c_0\in\LOne$;
  \item
  $0\leq v_0\leq1,\ v_0\notequiv 0,1,\ v_0^{\frac{1}{2}}\in \HOne$.
\end{enumerate}
\end{Assumption}

\noindent
 The major challenge of model \eqref{hapto} lies in the fact that the diffusion coefficient in equation \eqref{c} degenerates 
 at $c=0$ and, moreover, at $v=0$. The latter seems to make it impossible to obtain an a priori estimate for the gradient of  $\varphi(c)$ in some 
 Lebesgue space for any  smooth, strictly increasing function  $\varphi$.  As a workaround, we are forced to consider an auxiliary function $\ln\left(1+v^{\frac{1}{2}}c\right)$ 
 involving {\it both} $c$ and $v$ and whose gradient we are able to estimate.\\
\noindent
This leads us to the following definition of weak solutions to  \eqref{hapto}: 
\begin{Definition}[Weak solution]\label{Defweak}
Let $c_0,v_0$ satisfy {\it Assumptions~\ref{ini}}.
We call a pair of functions $c:\R^+_0\times\overline{\Omega}\rightarrow\R^+_0$, $v:\R^+_0\times\overline{\Omega}\rightarrow[0,1]$ a global weak solution of
 \eqref{hapto} if for all $0<T<\infty$ it holds that
\begin{enumerate}
  \item $c\in L^2(0,T;\LTwo)$, $\partial_t c\in L^1(0,T;(W^{1,\infty}(\Omega))')$;
  \item $v^{\frac{1}{2}}\in L^{\infty}(0,T;\HOne)$, $\partial_t v^{\frac{1}{2}}\in L^2(0,T;\LTwo)$;
  \item $\ln\left(1+v^{\frac{1}{2}}c\right)\in L^{\frac{4}{3}}(0,T;W^{1,\frac{4}{3}}(\Omega))$, $\frac{ v^{\frac{1}{2}}c}{1+vc}\left(\left(1+v^{\frac{1}{2}}c\right)\nabla \ln\left(1+v^{\frac{1}{2}}c\right)-c\nabla v^{\frac{1}{2}}\right)\in L^1(0,T;\LOne)$;
  \item $(c,v)$ satisfies equation \eqref{c} and the boundary condition \eqref{bc} in the following weak sense: 
  \begin{align}
   \left<\partial_t c,\varphi\right>=&-\left(\frac{\kappa_c v^{\frac{1}{2}}c}{1+vc}\left(\left(1+v^{\frac{1}{2}}c\right)\nabla \ln\left(1+v^{\frac{1}{2}}c\right)-c\nabla v^{\frac{1}{2}}\right)-\frac{\kappa_v c}{(1+v)^2}\nabla v,\nabla\varphi\right)\nonumber\\
   &+(\mu_cc(1-c-\eta v),\varphi)\text{ a.e.  in }(0,T)\text{ for all }\varphi\in W^{1,\infty}(\Omega);\nonumber
  \end{align}
  \item $(c,v)$ satisfies equation \eqref{v} a.e. in $(0,T)\times\Omega$;
  \item $c(0)=c_0$, $v(0)=v_0$.
 \end{enumerate}
\end{Definition}
\begin{Remark}[Diffusion term]
 If $c\in L^1(\tau,T;W^{1,1}(O))$ for some numbers $0\leq\tau<T$ and  open set $O\subset\Omega$, then due to the weak chain and product rules it holds that
 \begin{align}
  \frac{\kappa_cv^{\frac{1}{2}}c}{1+vc}\left(\left(1+v^{\frac{1}{2}}c\right)\nabla \ln\left(1+v^{\frac{1}{2}}c\right)-c\nabla v^{\frac{1}{2}}\right)=\frac{\kappa_c vc}{1+vc}\nabla c\text{ in }L^1(\tau,T;L^1(O)).\nonumber
 \end{align}
\end{Remark}
\begin{Remark}[Initial conditions]
Since we are looking for solutions $(c,v)$ with \begin{align*}
&c\in W^{1,1}((0,T);W^{-1,1}(\Omega)),\\
&v^{\frac{1}{2}}\in H^1(0,T;\LTwo),                                                 \end{align*}
we have 
\begin{align*}
&c\in C([0,T];W^{-1,1}(\Omega)),\\
&v^{\frac{1}{2}}\in C([0,T];\LTwo).                                                 \end{align*}
Therefore, the initial conditions 6. in Definition \ref{Defweak} do make sense.
\end{Remark}
\noindent
Our main result reads:
\begin{Theorem}[Global existence]\label{maintheo}
Let $\Omega\subset\R^N$, $N\in\N$, be a smooth bounded domain and let $\kappa, \mu_c, \eta, \mu_v, \lambda$ be positive constants. 
 Then, for each pair of functions $(c_0,v_0)$ satisfying {\it Assumptions~\ref{ini}}
 there exists a global weak solution $(c,v)$ (in terms of {\it Definition~\ref{Defweak}})  to the system \eqref{hapto}.
\end{Theorem}
\noindent
The proof of {\it Theorem~\ref{maintheo}} is based on a suitable approximation of the degenerate PDE-ODE system \eqref{hapto} by a family of nondegenerate PDE-PDE systems, 
derivation of a set of priori estimates which ensure necessary compactness and, finally, the passing to the limit. 
While the overall structure of the proof  
is a standard one for a haptotaxis system,  we encounter considerable difficulties in each of the three steps due to the previously mentioned degenerate diffusion 
in equation \eqref{c}, due to the ODE \eqref{v} having  no diffusion at all (i.e., everywhere degenerate), and, finally, due to a strong coupling between the two equations. 


\begin{Remark}[Notation]
We make the following useful convention: 
The statement that a constant depends on the parameters of the problem means that it depends on the constants $\kappa, \mu_c, \eta, \mu_v, \lambda$ and $\theta$ (see below), the 
norms of the initial data $(c_0,v_0)$, the space dimension $N$, and the domain $\Omega$. This dependence on the parameters is subsequently {\bf not} indicated in an explicit way.
\end{Remark}
\section{Approximating problems}\label{aproxi}
In this section, we introduce and study a family of non-degenerate approximations for problem \eqref{hapto}. However, before adding some regularizing terms to the system, we 
reformulate it in a manner that turns out to be convenient for our analysis:
\begin{subequations}\label{haptomod}
\begin{alignat}{3}
 &\partial_t c=\nabla\cdot\left(\frac{\kappa_c vc}{1+vc}\nabla c-\frac{2\kappa_v v^{\frac{1}{2}}c}{1+v}\nabla \psi(v)\right)+\mu_cc(1-c-\eta v)&&
 \text{ in }\R^+\times\Omega,\label{cmod}\\
&\partial_t \psi(v)=\frac{\mu_v v^{\frac{1}{2}}(1-v)-\lambda v^{\frac{1}{2}}c}{1+v}&&\text{ in }\R^+\times\Omega, \label{v2}\\
 &\frac{\kappa_c vc}{1+vc}\partial_{\nu} c-\frac{2\kappa_v v^{\frac{1}{2}}c}{1+v}\partial_{\nu} \psi(v)=0&&\text{ in }\R^+\times\partial\Omega,\label{bcmod}\\
 &c(0)=c_0,\ \psi(v(0))=\psi(v_0) &&\text{ in }\Omega,
\end{alignat}
\end{subequations}
where $$\psi:[0,1]\rightarrow\left[0,\frac{\pi}{4}\right],\  \psi(v):=\frac{1}{2}\int_0^{v}\frac{1}{s^{\frac{1}{2}}(1+s)}\,ds=\arctan\left(v^{\frac{1}{2}}\right).$$ 
Unlike the model in \cite{SSW}, the haptotaxis coefficient lacks a factor $v$ in the nominator, whose presence was essential for obtaining estimates involving  
$\nabla v$ (and which relied on differentiating the equation for $v$). Here we compensate the absence of $v$ by rearranging equation \eqref{v} in a convenient way.\\

\noindent
Equation \eqref{v2} is obtained from \eqref{v} by dividing both sides of the equation by $v^{\frac{1}{2}}(1+v)$. For the taxis part of the flux, we used the obvious identity
\begin{align}
 \frac{2\kappa_v v^{\frac{1}{2}}c}{1+v}\nabla \psi(v)=\frac{\kappa_v c}{(1+v)^2}\nabla v.\nonumber
\end{align}
Clearly, $\psi$ is a strictly monotonically increasing function and satisfies
\begin{align}
 \frac{1}{2} \left(v^{\frac{1}{2}}\right)'\leq\psi'(v)\leq  \left(v^{\frac{1}{2}}\right)' \text{ for all }v\in[0,1].\label{dpsi}
\end{align}
Starting from  \eqref{haptomod}, we fix some $$\theta>N+2$$ and consider for each  $\epsilon=\left(\epsilon_{1},\epsilon_{2},\epsilon_{3},\epsilon_{4}\right)\in(0,1)^4$ the  
system 
\begin{subequations}\label{chemoe}
\begin{alignat}{3}
  &\partial_t c_{\epsilon}=\epsilon_{2}\Delta c_{\epsilon}+\nabla\cdot\left(\frac{\kappa_c v_{\epsilon}c_{\epsilon}}{1+v_{\epsilon}c_{\epsilon}}\nabla c_{\epsilon}-\frac{2\kappa_v v_{\epsilon}^{\frac{1}{2}}c_{\epsilon}}{1+v_{\epsilon}}\nabla \psi(v_{\epsilon})\right)+\mu_c c_{\epsilon} (1-c_{\epsilon}-\eta v_{\epsilon})-\epsilon_{1} c_{\epsilon}^{\theta}&&\text{ in }\R^+\times\Omega,\label{eq1e}\\
 &\partial_t \psi(v_{\epsilon})=\epsilon_{1}\Delta \psi(v_{\epsilon}) +\frac{\mu_v v_{\epsilon}^{\frac{1}{2}}(1-v_{\epsilon})-\lambda v_{\epsilon}^{\frac{1}{2}}c_{\epsilon}}{1+v_{\epsilon}}&&\text{ in }\R^+\times\Omega,\label{eq2e}\\
 &\partial_{\nu} c_{\epsilon}=\partial_{\nu} \psi(v_{\epsilon})=0&&\text{ in }\R^+\times\partial\Omega,\label{bc1}\\
 &c_{\epsilon}(0)=c_{\epsilon_{3}0},\ \psi(v_{\epsilon}(0))=\psi(v_{\epsilon_{4}0}) &&\text{ in }\Omega.
\end{alignat}
\end{subequations}
Here, the families $\{c_{\epsilon_{3}0}\}$ and $\{v_{\epsilon_{4}0}\}$ of initial values are parameterized by $\epsilon_{3}$ and $\epsilon_{4}$, respectively, 
are independent of $\epsilon_{1},\epsilon_{2}$ and  satisfy
\begin{align}
 &c_{\epsilon_{3}0},\psi(v_{\epsilon_{4}0})\in W^{1,\infty}(\Omega),\nonumber\\
 &c_{\epsilon_{3}0}\geq0,\ 0\leq v_{\epsilon_{4}0}\leq1 \text{ in }\overline{\Omega},\ c_{\epsilon_{3}0} \notequiv0,\ v_{\epsilon_{4}0}\notequiv0,1,\nonumber\\
 &||c_{\epsilon_{3}0}\ln c_{\epsilon_{3}0}||_1,\ \left\|\psi(v_{\epsilon_{4}0})\right\|_{H^1(\Omega)}\leq\C.\nonumber
\end{align}
They are yet to be further specified below. 

\noindent
For each $\epsilon_{1},\epsilon_{2}\in(0,1)$, system the \eqref{eq1e}-\eqref{eq2e} has the form of a nondegenerate quasilinear {\it chemotaxis} system with respect to variables $c_{\epsilon}$ and $\psi(v_{\epsilon})$. It is clear that for $\epsilon=0$ we regain - at least formally - the original degenerate haptotaxis system \eqref{cmod}-\eqref{v2}.
As it turns out (see the subsequent {\it Section~\ref{existence}}), a weak solution to \eqref{haptomod} can be obtained as a limit of a sequence of solutions to \eqref{chemoe}.\\ 

\noindent
In order to prove the global well-posedness for system \eqref{chemoe}, we intend to use the standard Amann theory for abstract parabolic quasilinear systems \cite{Amann1}. 

\noindent
We need some more notations. Let us define for all $(c,v)\in\R^+_0\times[0,1]$ the matrices 
\begin{align*}
 &A_{\epsilon}(c,v):=\left[\begin{array}{cc}
          \epsilon_{2}+\frac{\kappa_c vc}{1+vc}&-\frac{2\kappa_v cv^{\frac{1}{2}}}{1+v}\\
           0&\epsilon_{1}
          \end{array}
\right],\quad \overline{A}_{\epsilon}(c,\psi(v)):=A_{\epsilon}(c,v),\\
&F_{\epsilon}(c,v):=\left[\begin{array}{c}
           \mu_c c (1-c-\eta v)-\epsilon_{1} c^{\theta}\\
           \frac{\mu_v v^{\frac{1}{2}}(1-v)-\lambda v^{\frac{1}{2}}c}{1+v}
          \end{array}
\right],\quad \overline{F}_{\epsilon}(c,\psi(v)):=F_{\epsilon}(c,v).
\end{align*}
Since $\psi:[0,1]\rightarrow\left[0,\frac{\pi}{4}\right]$ is a strictly monotonically increasing function, $\overline{A}_{\epsilon}$ and $\overline{F}_{\epsilon}$ are well-defined on $\R^+_0\times\left[0,\frac{\pi}{4}\right]$. Let also
\begin{align*}
 V_{\epsilon}:=\psi(v_{\epsilon}),\ V_{\epsilon0}:=\psi(v_{\epsilon_{4}0}).
\end{align*}
In this notation, system \eqref{chemoe} takes the form
\begin{subequations}\label{chemoemod}
\begin{alignat}{3}
 &\partial_t c_{\epsilon}=\nabla\cdot\left([\overline{A}_{\epsilon}(c_{\epsilon},V_{\epsilon})]_{11}\nabla c_{\epsilon} +[\overline{A}_{\epsilon}(c_{\epsilon},V_{\epsilon})]_{12}\nabla V_{\epsilon})\right)+[\overline{F}_{\epsilon}(c_{\epsilon},V_{\epsilon})]_1&&\text{ in }\R^+\times\Omega,\label{eq1emod}\\
 &\partial_t V_{\epsilon}=\nabla\cdot\left([\overline{A}_{\epsilon}(c_{\epsilon},V_{\epsilon})]_{21}\nabla c_{\epsilon} +[\overline{A}_{\epsilon}(c_{\epsilon},V_{\epsilon})]_{22}\nabla V_{\epsilon})\right)+[\overline{F}_{\epsilon}(c_{\epsilon},V_{\epsilon})]_2&&\text{ in }\R^+\times\Omega,\\
  &\partial_{\nu} c_{\epsilon}=\partial_{\nu} V_{\epsilon}=0&&\text{ in }\R^+\times\partial\Omega,\label{bc1mod}\\
 &c_{\epsilon}(0)=c_{\epsilon_{3}0},\ V_{\epsilon}(0)=V_{\epsilon0} &&\text{ in }\Omega.
\end{alignat}
\end{subequations}
It is easy to see that
\begin{enumerate}
 \item $\overline{A}_{\epsilon}$ and $\overline{F}_{\epsilon}$ are (infinitely) smooth;
 \item $\overline{A}_{\epsilon}$ is upper triangular  with $\left[\overline{A}_{\epsilon}\right]_{22}$  independent of the first variable;
 \item $\left[\overline{A}_{\epsilon}\right]_{11}\geq \epsilon_{2}>0, 
 \left[\overline{A}_{\epsilon}\right]_{22}\geq\epsilon_{1}>0$;
 \item $c_{\epsilon_{3}0}\geq0$, $0\leq V_{\epsilon0}\leq\frac{\pi}{4}$, $c_{\epsilon_{3}0}\notequiv0$, $V_{\epsilon0}\notequiv0,\frac{\pi}{4}$. 
\end{enumerate}
In this situation, we may apply several results from \cite{Amann1} on local and global existence and regularity of solutions for regular quasilinear parabolic systems, 
see \cite[Theorems 14.4, 14.7,  and 15.5]{Amann1}. These results yield the following:

\noindent
if  for all $0<\tau<T$ it holds  {\it a priori} that
\begin{align}
&0<c_{\epsilon}\leq C\left(\epsilon_{1}^{-1},\epsilon_{2}^{-1},\tau^{-1},T\right),\ 0<V_{\epsilon}<\frac{\pi}{4}\text{ in }[\tau,T]\times\overline{\Omega                                                                                                                                                                                                                                                                                                      },\label{bnd1}
\end{align}
then problem \eqref{chemoemod} has a unique global classical nonnegative solution $(c_{\epsilon},V_{\epsilon})$, and this solution satisfies \eqref{bnd1}. 

\noindent
Thus, it remains to prove that  \eqref{bnd1} holds a priori. Observe first that 
\begin{enumerate}[resume]
 \item  $\left[\overline{A}_{\epsilon}(0,\cdot)\right]_{12},\left[\overline{A}_{\epsilon}(\cdot,0)\right]_{21}\equiv0$ and $[\overline{F}_{\epsilon}(0,\cdot)]_1,[\overline{F}_{\epsilon}(\cdot,0)]_2=0$;
 \item  $\left[\overline{A}_{\epsilon}\left(\cdot,\frac{\pi}{4}\right)\right]_{21}\equiv0$ and $[\overline{F}_{\epsilon}\left(\cdot,\frac{\pi}{4}\right)]_2\leq 0$.
\end{enumerate}
Hence, due to the strong maximum principle and the Hopf lemma for parabolic equations, it holds a priori  that
\begin{align}
&c_{\epsilon}> 0,\ 0< V_{\epsilon}<\frac{\pi}{4}\text{ in }\R^+\times\overline{\Omega},\nonumber
\end{align}
or, in terms of the original variables,
\begin{align}
&c_{\epsilon}> 0,\ 0< v_{\epsilon}<1\text{ in }\R^+\times\overline{\Omega}.\label{hopf}
\end{align}
Next, we integrate both sides of \eqref{eq1e}  over $\Omega$, using partial integration and the boundary conditions where necessary. We obtain that
\begin{align}
 \frac{d}{dt}||c_{\epsilon}||_1=&\int_{\Omega}\mu_c c_{\epsilon} (1-c_{\epsilon}-\eta v_{\epsilon})-\epsilon_{1} c_{\epsilon}^{\theta}\,dx\nonumber\\
 \leq&\C-\C||c_{\epsilon}||_1-\epsilon_{1}\int_{\Omega}c_{\epsilon}^{\theta}\,dx.\label{est5}
\end{align}
Estimate \eqref{est5} yields with help of the Gronwall lemma that
\begin{align}
  &||c_{\epsilon}||_{L^{\infty}(\R^+_0;\LOne)}\leq\C,\label{c1}\\
 & ||c_{\epsilon}||_{L^{\theta}((0,T)\times\Omega)}\leq\C\left(\epsilon_{1}^{-1},T\right).\label{a7_1}
\end{align}
Combining \eqref{hopf} and \eqref{a7_1}, we conclude from \eqref{eq2e} that 
\begin{align}
 \left\|\partial_t \psi(v_{\epsilon})-\epsilon_{2}\Delta \psi(v_{\epsilon})\right\|_{L^{\theta}((0,T)\times\Omega)}\leq\C\left(\epsilon_{1}^{-1},T\right).\label{est10}
\end{align}
Together with known results on maximal Sobolev regularity for parabolic equations (compare, e.g., Theorems 4.10.2 and 4.10.7 and Remark 4.10.9 from \cite{Amann1995}), \eqref{est10} yields that 
\begin{align}
 ||\psi(v_{\epsilon})||_{C([\tau,T];W^{2\left(1-\frac{1}{\theta}\right),\theta}(\Omega))}\leq\C\left(\epsilon_{1}^{-1},\epsilon_{2}^{-1},\tau^{-1},T\right),\ \tau\in(0,T].\nonumber
\end{align}
Using the Sobolev embedding $W^{2\left(1-\frac{1}{\theta}\right),\theta}(\Omega)\subset \LInf$ (recall that $\theta>N+2$), we thus arrive at 
\begin{align}
 ||\nabla\psi(v_{\epsilon})||_{L^{\infty}((\tau,T)\times\Omega)}\leq\Cl{q}\left(\epsilon_{1}^{-1},\epsilon_{2}^{-1},\tau^{-1},T\right),\ \tau\in(0,T].\label{npsi}
\end{align}

\noindent
Let us now return to equation \eqref{eq1e}. It can be rewritten in the form
\begin{align}
 \partial_t c_{\epsilon}=\nabla\cdot q_{\epsilon}+f_{\epsilon},\label{eqd}
\end{align}
where 
\begin{align}
&q_{\epsilon}:=a_{\epsilon}\nabla c_{\epsilon} +Q_{\epsilon}c_{\epsilon},\quad 
a_{\epsilon}:=\epsilon_{2}+\frac{\kappa_c v_{\epsilon}c_{\epsilon}}{1+v_{\epsilon}c_{\epsilon}},\quad 
 Q_{\epsilon}:=-\frac{2\kappa_v v_{\epsilon}^{\frac{1}{2}}}{1+v_{\epsilon}}\nabla \psi(v_{\epsilon}),\label{qe}\\
 & f_{\epsilon}:=\mu_c c_{\epsilon} (1-c_{\epsilon}-\eta v_{\epsilon})-\epsilon_{1} c_{\epsilon}^{\theta}.\label{fe}
\end{align}
Equation \eqref{eqd} is in divergence form. 
Due to \eqref{hopf} and \eqref{npsi}, its coefficients satisfy the inequalities
\begin{align}
a_{\epsilon}\geq\epsilon_{2},\quad 
 ||Q_{\epsilon}||_{L^{\infty}((\tau,T)\times\Omega)}\leq2\kappa_v\Cr{q}\left(\epsilon_{1}^{-1},\epsilon_{2}^{-1},\tau^{-1},T\right),\ \tau\in(0,T],\quad
 f_{\epsilon}\leq\C.\nonumber
\end{align}
Therefore, standard results on uniform boundedness for linear parabolic equations \cite[Chapter 3, \S 7]{LSU} are applicable to equation \eqref{eqd} equipped with homogeneous Neumann boundary conditions and yield
\begin{align}
 ||c_{\epsilon}||_{L^{\infty}((\tau,T)\times\Omega)}\leq\C\left(\epsilon_{1}^{-1},\epsilon_{2}^{-1},\tau^{-1},T\right),\ \tau\in(0,T].\nonumber
\end{align}
This finishes the proof of \eqref{bnd1}.

\subsection*{Approximating initial data}
Our next step is to construct a suitable family of approximations to the initial data. Since we assume that $(c_0,v_0)$ satisfies {\it Assumptions~\ref{ini}}, there exists for each $\left(\epsilon_{3},\epsilon_{4}\right)\in(0,1)^2$ a pair of approximations $\left(c_{\epsilon_{3}0},v_{\epsilon_{4}0}\right)$ with the following properties:
\begin{align}
 &c_{\epsilon_{3}0},v_{\epsilon_{4}0}^{\frac{1}{2}}\in W^{1,\infty}(\Omega),\label{e34bnd}\\
 &c_{\epsilon_{3}0}\geq0,\ 0\leq v_{\epsilon_{4}0}\leq1 \text{ in }\overline{\Omega},\ c_{\epsilon_{3}0}, v_{\epsilon_{4}0}\notequiv0,\\
 &||c_{\epsilon_{3}0}\ln c_{\epsilon_{3}0}||\leq2||c_{0}\ln c_{0}||_1,\\
 & \left\|\nabla v_{\epsilon_{4}0}^{\frac{1}{2}}\right\|\leq2\left\| v_0^{\frac{1}{2}}\right\|_{H^1(\Omega)},\\
 &||c_{\epsilon_{3}0}-c_{0}||_1\leq\epsilon_{3},\label{coe}\\
 & \left\|v_{\epsilon_{4}0}^{\frac{1}{2}}-v_{0}^{\frac{1}{2}}\right\|\leq\epsilon_{4}.\label{voe}
\end{align}
Recall our aim is to pass to the limit for $\epsilon \to 0$ in the approximating problem. After letting $\epsilon _1\to 0$ in equation \eqref{eq2e} we obtain an ODE,  
hence the set $\{v(t,\cdot )=0\}$ is preserved in time (possibly up to some subsets of measure zero). Therefore, it turns out that we have to pay particular care at the set $\{v_{\epsilon_40}=0\}$ which should not shrink substantially with respect to $\{v_0=0\}$. We may assume that
\begin{align}
 \left|\{v_0=0\}\backslash\inner\left\{v_{\epsilon_{4}0}=0\right\}\right|\leq\epsilon_{4}.\label{star}
\end{align}
Indeed, due to a Lusin property for Sobolev functions  \cite[Chapter 6, Theorem 6.14]{EvansGar}, there exists a function $\xi$ such that
\begin{align}
&\xi\in W^{1,\infty}(\Omega),\label{L1}\\
 & \left\|\xi\right\|_{H^1(\Omega)}\leq2\left\|v_0^{\frac{1}{2}}\right\|_{H^1(\Omega)},\label{L12}\\
 &\left|\left\{\xi\neq v_0^{\frac{1}{2}}\right\}\right|\leq\frac{\epsilon_{4}}{4}.\label{L2}
\end{align}
We define  $$v_{\epsilon_{4}0}:=\left(\min\{\xi_+,1\}-\frac{\epsilon_{4}}{2|\Omega|}\right)_+^2.$$ Let us check that $v_{\epsilon_{4}0}$ satisfies the above assumptions.  Indeed, due to \eqref{L1}-\eqref{L12}, we have that
\begin{align}
&v_{\epsilon_{4}0}^{\frac12}\in W^{1,\infty}(\Omega),\\
 &\left\|\nabla v_{\epsilon_{4}0}^{\frac{1}{2}}\right\|\leq ||\nabla \xi||\leq2\left\|v_0^{\frac{1}{2}}\right\|_{H^1(\Omega)},\nonumber
 \end{align}
and
\begin{align}
 \left\|v_{\epsilon_{4}0}^{\frac{1}{2}}-v_{0}^{\frac{1}{2}}\right\| \leq &2\left|\left\{\xi\neq v_0^{\frac{1}{2}}\right\}\right|+\left\|\chi_{\left\{\xi= v_0^{\frac{1}{2}}\right\}}\left(\left(\xi-\frac{\epsilon_{4}}{2|\Omega|}\right)_+-\xi\right)\right\|\nonumber\\
 \leq&\epsilon_{4}.
 \end{align}
Moreover, it holds that 
 \begin{align}
 &\{\xi=0\}\subset\left\{\min\{\xi_+,1\}<\frac{\epsilon_{4}}{2|\Omega|}\right\}\subset\inner\left\{\min\{\xi_+,1\}\leq\frac{\epsilon_{4}}{2|\Omega|}\right\}\cup\partial\Omega=\inner\{v_{\epsilon_{4}0}=0\}\cup\partial\Omega.\label{L4}
\end{align}
Combining \eqref{L2} and \eqref{L4}, we obtain  \eqref{star}.
\section{A priori estimates}\label{apriori}
In this section we establish several uniform a priori estimates for system \eqref{chemoe}. 
To begin with, we apply the gradient operator to both sides of \eqref{eq2e}:
\begin{align}
\partial_t \nabla \psi(v_{\epsilon})=&\epsilon_{1}\Delta \nabla \psi(v_{\epsilon})
-\lambda \frac{v_{\epsilon}^{\frac{1}{2}}}{1+v_{\epsilon}}\nabla c_{\epsilon}-\frac{\lambda (1-v_{\epsilon})c_{\epsilon}+\mu_v(-1+4v_{\epsilon}+v_{\epsilon}^2)}{(1+v_{\epsilon})^2}\nabla v_{\epsilon}^{\frac{1}{2}}.
 \label{deq2e}
\end{align}
Further, we multiply \eqref{eq1e} by 
$\ln c_{\epsilon}$ and \eqref{deq2e}  by $\frac{\kappa_v}{\lambda}\nabla \psi(v_{\epsilon})$ and integrate over $\Omega$ using partial integration and the boundary conditions where necessary. Adding the resulting identities together, we obtain after some calculation that
\begin{align}
 &\frac{d}{dt}\left(\left(1,c_{\epsilon}\ln c_{\epsilon}-c_{\epsilon}\right)+\frac{2\kappa_v}{\lambda}\left(\frac{1}{(1+v_{\epsilon})^2},\left|\nabla v_{\epsilon}^{\frac{1}{2}}\right|^2\right)\right)+\epsilon_{2}\left(\frac{1}{c_{\epsilon}},|\nabla c_{\epsilon}|^2\right)+\epsilon_{1}\left(\frac{1}{\theta}\left(c_{\epsilon}^{\theta},\ln c_{\epsilon}^{\theta}\right)+\left\|\Delta \psi(v_{\epsilon})\right\|^2\right)\nonumber\\
 &+\left(\frac{\kappa_c v_{\epsilon}}{1+v_{\epsilon}c_{\epsilon}},|\nabla c_{\epsilon}|^2\right)+\frac{2\kappa_v}{\lambda}\left(\lambda(1-v_{\epsilon})c_{\epsilon}+5\mu_vv_{\epsilon}+\mu_vv_{\epsilon}^2,\frac{\left|\nabla v_{\epsilon}^{\frac{1}{2}}\right|^2}{(1+v_{\epsilon})^3}\right)\nonumber\\
 \leq&\left(\mu_c c_{\epsilon} (1-c_{\epsilon}-\eta v_{\epsilon}),\ln c_{\epsilon} \right)+\frac{2\mu_v\kappa_v}{\lambda}\left(\frac{1}{(1+v_{\epsilon})^2},\left|\nabla v_{\epsilon}^{\frac{1}{2}}\right|^2\right)\nonumber\\
 \leq&-\Cl{C4}\left(\chi_{\{c_{\epsilon}>1\}},c_{\epsilon}^2\ln c_{\epsilon}\right)+\frac{2\mu_v\kappa_v}{\lambda}\left(\frac{1}{(1+v_{\epsilon})^2},\left|\nabla v_{\epsilon}^{\frac{1}{2}}\right|^2\right)+\Cl{C5}.
\end{align}
By using the Gronwall lemma, we thus arrive for arbitrary $T\in\R^+$ at the estimates
\begin{align}
& \max_{t\in[0,T]}\left(\chi_{\{c_{\epsilon}>1\}},c_{\epsilon}\ln c_{\epsilon}\right)\leq\Cr{C2}(T),\label{a1}\\ & \max_{t\in[0,T]}\left\|\nabla v_{\epsilon}^{\frac{1}{2}}\right\|^2\leq\Cl{C2}(T),\label{a2}\\
&\int_0^T\left\|c_{\epsilon}^2\ln c_{\epsilon}^2\right\|_1\,dt\leq\Cr{C2}(T),\label{a3}\\ & \int_0^T\left(\frac{v_{\epsilon}}{1+v_{\epsilon}c_{\epsilon}},|\nabla c_{\epsilon}|^2\right)\,dt\leq\Cr{C2}(T),\label{a4}\\ 
& \int_0^T\left((1-v_{\epsilon})c_{\epsilon},\left|\nabla v_{\epsilon}^{\frac{1}{2}}\right|^2\right)\,dt\leq\Cr{C2}(T),\label{a5}\\
&\int_0^T\left(\frac{1}{c_{\epsilon}},|\nabla c_{\epsilon}|^2\right)\,dt\leq\epsilon_{2}^{-1}\Cr{C2}(T),\label{a6}\\ 
& \int_0^T\left\|c_{\epsilon}^{\theta}\ln c_{\epsilon}^{\theta}\right\|_1\,dt\leq\epsilon_{1}^{-1}\Cr{C2}(T),\label{a7}\\ 
& \int_0^T\left\|\Delta \psi(v_{\epsilon})\right\|^2\,dt\leq\epsilon_{1}^{-1}\Cr{C2}(T).\label{a8}
\end{align}
Throughout the section, we will obtain further estimates for functions  $c_{\epsilon}$ and $v_{\epsilon}$ and their combinations, which we will use in the existence proof (see {\it Section~\ref{existence}} below). 

\noindent
By means of the de la Vall\'ee-Poussin theorem, we obtain from \eqref{a3} that 
\begin{align}
 \left\{c_{\epsilon}^2\right\}\text{ is uniformly integrable in } (0,T)\times\Omega\label{ui}
\end{align}
with
\begin{align}
 ||c_{\epsilon}||_{L^2((0,T)\times\Omega)}\leq\C(T).  \label{c2}
\end{align}

\noindent
Next, we deal with the relaxation terms in \eqref{eq1e}. Using the H\"older inequality, we obtain with  \eqref{a6} and \eqref{c2} that
\begin{eqnarray}
 \left\|\epsilon_{2}\nabla c_{\epsilon}\right\|_{L^{\frac{4}{3}}((0,T)\times\Omega)}&\leq &\epsilon_{2}||c_{\epsilon}||_{L^2((0,T)\times\Omega)}^{\frac{1}{2}}\left(\int_0^T\left(\frac{1}{c_{\epsilon}},|\nabla c_{\epsilon}|^2\right)\,dt\right)^{\frac{1}{2}}\nonumber\\
 &\leq &\epsilon_{2}^{\frac{1}{2}}\Cl{C9}(T)\label{nce}\\
 &\underset{\epsilon_{2}\rightarrow0}{\rightarrow}&0,\label{este2}
\end{eqnarray}
Further, since the function $g(y):=y\ln y$ is convex, we obtain from \eqref{a7} with help of the Jensen's inequality that
\begin{align}
 g\left(\frac{1}{|(0,T)\times\Omega|}\int_{(0,T)\times\Omega}\max\left\{1,c_{\epsilon}^{\theta}\right\}\,dxdt\right)\leq&\frac{1}{|(0,T)\times\Omega|}\int_{(0,T)\times\Omega}g\left(\max\left\{1,c_{\epsilon}^{\theta}\right\}\right)\,dxdt\nonumber\\
 \leq& \frac{\Cl{C10}(T)}{\epsilon_{1}|(0,T)\times\Omega|}.\label{este3}
\end{align}
Since $g$ is increasing on $(1,\infty)$ and $\frac{g(y)}{y}\underset{y\rightarrow\infty}{\rightarrow}\infty$, \eqref{este3} yields that
\begin{eqnarray}
 \epsilon_{1}\int_{(0,T)\times\Omega} c_{\epsilon}^{\theta}\,dxdt&\leq &\epsilon_{1}|(0,T)\times\Omega|g^{(-1)}\left(\frac{\Cr{C10}(T)}{\epsilon_{1}|(0,T)\times\Omega|}\right)\nonumber\\
 &\underset{\epsilon_{1}\rightarrow0}{\rightarrow}&0.\label{este4}
\end{eqnarray}
Using \eqref{hopf}, we estimate the reaction term $f_{\epsilon}$  (as defined in \eqref{fe}):
\begin{align}
 |f_{\epsilon}|=\left|\mu_c c_{\epsilon} (1-c_{\epsilon}-\eta v_{\epsilon})-\epsilon_{1} c_{\epsilon}^{\theta}\right|\leq \C\left(c_{\epsilon}^2+1\right)+\epsilon_{1} c_{\epsilon}^{\theta}.\nonumber
\end{align}
Hence, due to \eqref{c2} and \eqref{este4}, it holds that
\begin{align}
||f_{\epsilon}||_{L^1((0,T)\times\Omega)}\leq\C(T).\label{f1}                                                                               \end{align}
Using \eqref{dpsi} and \eqref{hopf}, we obtain from \eqref{eq2e} that
\begin{align}
 \left|\partial_t \psi(v_{\epsilon})\right|\leq&2\left|\partial_t v_{\epsilon}^{\frac{1}{2}}\right|\nonumber\\
 \leq&2\epsilon_{1}|\Delta \psi(v_{\epsilon})| +2\left|\frac{\mu_v v_{\epsilon}^{\frac{1}{2}}(1-v_{\epsilon})-\lambda v_{\epsilon}^{\frac{1}{2}}c_{\epsilon}}{1+v_{\epsilon}}\right|\nonumber\\
 \leq &2\epsilon_{1}|\Delta \psi(v_{\epsilon})| +2\mu_v+2\lambda c_{\epsilon}. \label{vt_}
\end{align}
Combining \eqref{a8} and \eqref{c2}, we conclude from \eqref{vt_} that 
\begin{align}
 \left\|\partial_t v_{\epsilon}^{\frac{1}{2}}\right\|_{L^2((0,T)\times\Omega)}\leq\C(T).\label{vt}
\end{align}
Next, we study the function $v_{\epsilon}^{\frac{1}{2}}c_{\epsilon}$.
Observe that
\begin{align}
 \left(v_{\epsilon}^{\frac{1}{2}}c_{\epsilon}\right)^2\ln\left(v_{\epsilon}^{\frac{1}{2}}c_{\epsilon}\right)^2=c_{\epsilon}^2v_{\epsilon}\ln v_{\epsilon}+v_{\epsilon}c_{\epsilon}^2\ln c_{\epsilon}^2,\nonumber
\end{align}
Hence, due to \eqref{hopf} and \eqref{a3}, it holds that
\begin{align}
 \left\|\left(v_{\epsilon}^{\frac{1}{2}}c_{\epsilon}\right)^2\ln\left(v_{\epsilon}^{\frac{1}{2}}c_{\epsilon}\right)^2\right\|_{L^{1}((0,T)\times\Omega)}\leq\C(T).\label{wui}
\end{align}
Again, we apply the de la Vall\'ee-Poussin theorem and obtain from \eqref{wui}  that
\begin{align}
 \left\{\left(v_{\epsilon}^{\frac{1}{2}}c_{\epsilon}\right)^2\right\}\text{ is uniformly integrable in }(0,T)\times\Omega\label{wui1}
\end{align}
with 
\begin{align}
 \left\|v_{\epsilon}^{\frac{1}{2}}c_{\epsilon}\right\|_{L^2((0,T)\times\Omega)}\leq\C(T). \label{vc2}
\end{align}
Next, we consider the degenerate part of the diffusion flux.
Using \eqref{hopf}, we estimate as follows:
\begin{align}
 \frac{v_{\epsilon}c_{\epsilon}|\nabla c_{\epsilon}|}{1+v_{\epsilon}c_{\epsilon}}=&\left(\frac{v_{\epsilon}c_{\epsilon}}{1+v_{\epsilon}c_{\epsilon}}\right)^{\frac{1}{2}}c_{\epsilon}^{\frac{1}{2}}\left(\frac{v_{\epsilon}}{1+v_{\epsilon}c_{\epsilon}}|\nabla c_{\epsilon}|^2\right)^{\frac{1}{2}}\nonumber\\
 \leq&c_{\epsilon}^{\frac{1}{2}}\left(\frac{v_{\epsilon}}{1+v_{\epsilon}c_{\epsilon}}|\nabla c_{\epsilon}|^2\right)^{\frac{1}{2}}.\label{dif1}
\end{align}
Combining \eqref{a4}, \eqref{c2} and \eqref{dif1}, we obtain with the H\"older inequality that 
\begin{align}
 \left\|\frac{v_{\epsilon}c_{\epsilon}|\nabla c_{\epsilon}|}{1+v_{\epsilon}c_{\epsilon}}\right\|_{L^{\frac{4}{3}}((0,T)\times\Omega)}\leq ||c_{\epsilon}||_{L^2((0,T)\times\Omega)}^{\frac12}\left(\int_0^T\left(\frac{v_{\epsilon}}{1+v_{\epsilon}c_{\epsilon}},|\nabla c_{\epsilon}|^2\right)\,dt\right)^{\frac{1}{2}}\leq\C(T).\label{dif2}
\end{align}
Using \eqref{hopf}, we also have that
\begin{align}
 \frac{v_{\epsilon}^{\frac{3}{4}}c_{\epsilon}|\nabla c_{\epsilon}|}{1+v_{\epsilon}c_{\epsilon}}=&\left(\frac{\left(v_{\epsilon}c_{\epsilon}\right)^{\frac{1}{2}}}{1+v_{\epsilon}c_{\epsilon}}\right)^{\frac{1}{2}}c_{\epsilon}^{\frac{3}{4}}\left(\frac{v_{\epsilon}}{1+v_{\epsilon}c_{\epsilon}}|\nabla c_{\epsilon}|^2\right)^{\frac{1}{2}}\nonumber\\
 \leq&c_{\epsilon}^{\frac{3}{4}}\left(\frac{v_{\epsilon}}{1+v_{\epsilon}c_{\epsilon}}|\nabla c_{\epsilon}|^2\right)^{\frac{1}{2}}.\label{dif1_}
\end{align}
Then, \eqref{a4}, \eqref{c2} and \eqref{dif1_}, together with the H\"older inequality, yield that 
\begin{align}
 \left\|\frac{v_{\epsilon}^{\frac{3}{4}}c_{\epsilon}|\nabla c_{\epsilon}|}{1+v_{\epsilon}c_{\epsilon}}\right\|_{L^{\frac{8}{7}}((0,T)\times\Omega)}\leq ||c_{\epsilon}||_{L^2((0,T)\times\Omega)}^{\frac{3}{4}}\left(\int_0^T\left(\frac{v_{\epsilon}}{1+v_{\epsilon}c_{\epsilon}},|\nabla c_{\epsilon}|^2\right)\,dt\right)^{\frac{1}{2}}\leq\C(T).\label{dif2_}
\end{align}
As for the taxis part of the flux, we combine \eqref{a2} and \eqref{vc2} with the H\"older inequality in order to obtain that   
\begin{align}
 \left\|\frac{2 v_{\epsilon}^{\frac{1}{2}}c_{\epsilon}\left|\nabla v_{\epsilon}^{\frac{1}{2}}\right|}{(1+v_{\epsilon})^2}
 \right\|_{L^1((0,T)\times\Omega)}\leq 2\left\|v_{\epsilon}^{\frac{1}{2}}c_{\epsilon}\right\|_{L^2((0,T)\times\Omega)}\left\|\nabla v_{\epsilon}^{\frac{1}{2}}\right\|_{L^2((0,T)\times\Omega)}\leq\C(T).\label{tax2}
\end{align}
Combining \eqref{este2}, \eqref{dif2} and \eqref{tax2}, we gain an estimate for the flux $q_{\epsilon}$ (as defined in \eqref{qe}):
\begin{align}
 ||q_{\epsilon}||_{L^1((0,T)\times\Omega)}\leq\C(T).\label{qe1_}
\end{align}
Together with \eqref{f1}, \eqref{qe1_} yields that
 \begin{align}
 ||\partial_t c_{\epsilon}||_{L^1(0,T;W^{-1,1}(\Omega))}\leq\C(T).\label{ct1}
\end{align}
\subsection*{Estimates for an auxiliary function} 
Owing to the fact that the original diffusion coefficient in \eqref{c} is degenerate in $v$, it does not seem possible to obtain a uniform (in $\epsilon$) estimate for the gradient of  $\varphi(c_{\epsilon})$ in some Lebesgue space over $(0,T)\times\Omega$ for any    smooth, strictly increasing, and independent of $\epsilon$ function  
$\varphi$. In order to overcome this difficulty, we introduce for $\epsilon\in(0,1)$ an auxiliary function which involves {\it both} $c_{\epsilon}$ and $v_{\epsilon}$:
\begin{align}
  u_{\epsilon}:=\ln\left(1+v_{\epsilon}^{\frac{1}{2}}c_{\epsilon}\right).\nonumber
\end{align}
With \eqref{hopf}, we have that
\begin{align}
 0\leq\ln\left(1+v_{\epsilon}^{\frac{1}{2}}c_{\epsilon}\right)\leq 1+c_{\epsilon},\nonumber
\end{align}
so that, due to  \eqref{c1}, it holds that
\begin{align}
 ||u_{\epsilon}||_{L^1((0,T)\times\Omega)}\leq\C(T).\nonumber
\end{align}
As it turns out, the family $\{u_{\epsilon}\}$ is precompact in $L^1((0,T)\times\Omega)$. To prove this, we need uniform estimates for the partial derivatives of $u_{\epsilon}$ in some parabolic Sobolev spaces.

\noindent
We first study the spatial gradient of $u_{\epsilon}$.
We compute that 
\begin{align}
 \nabla u_{\epsilon}=\frac{c_{\epsilon}}{1+v_{\epsilon}^{\frac{1}{2}}c_{\epsilon}}\nabla v_{\epsilon}^{\frac{1}{2}}+\frac{v_{\epsilon}^{\frac{1}{2}}}{1+v_{\epsilon}^{\frac{1}{2}}c_{\epsilon}}\nabla c_{\epsilon}.\label{gru}
\end{align}
Using \eqref{hopf} and the trivial inequality \begin{align}
1\leq v_{\epsilon}^{\frac{1}{2}}+(1-v_{\epsilon})^{\frac{1}{2}},   \label{triv}                                           \end{align}
 we estimate the first summand on the right-hand side of \eqref{gru} in the following way:
\begin{align}
 \frac{c_{\epsilon}\left|\nabla v_{\epsilon}^{\frac{1}{2}}\right|}{1+v_{\epsilon}^{\frac{1}{2}}c_{\epsilon}}\leq&\frac{v_{\epsilon}^{\frac{1}{2}}c_{\epsilon}\left|\nabla v_{\epsilon}^{\frac{1}{2}}\right|}{1+v_{\epsilon}^{\frac{1}{2}}c_{\epsilon}}+\frac{(1-v_{\epsilon})^{\frac{1}{2}}c_{\epsilon}\left|\nabla v_{\epsilon}^{\frac{1}{2}}\right|}{1+v_{\epsilon}^{\frac{1}{2}}c_{\epsilon}}\nonumber\\
 \leq&\left|\nabla v_{\epsilon}^{\frac{1}{2}}\right|+c_{\epsilon}^{\frac{1}{2}}\left((1-v_{\epsilon})c_{\epsilon}\left|\nabla v_{\epsilon}^{\frac{1}{2}}\right|^2\right)^{\frac{1}{2}}.
\label{estgr1}
\end{align}
Using the H\"older inequality and estimates \eqref{a2}, \eqref{a5}  and \eqref{c2}, we conclude from \eqref{estgr1} that
\begin{align}
 \left\|\frac{c_{\epsilon}\left|\nabla v_{\epsilon}^{\frac{1}{2}}\right|}{1+v_{\epsilon}^{\frac{1}{2}}c_{\epsilon}}\right\|_{L^{\frac{4}{3}}((0,T)\times\Omega)}\leq& \C(T)\left\|\nabla v_{\epsilon}^{\frac{1}{2}}\right\|_{L^2((0,T)\times\Omega)}+||c_{\epsilon}||_{L^2((0,T)\times\Omega)}^{\frac12}\left(\int_0^T\left((1-v_{\epsilon})c_{\epsilon},\left|\nabla v_{\epsilon}^{\frac{1}{2}}\right|^2\right)\,dt\right)^{\frac{1}{2}}\nonumber\\
 \leq&\C(T).\label{estgr1_}
\end{align}
For the second summand on the right-hand side of \eqref{gru}, we have that
\begin{align}
 &\frac{v_{\epsilon}^{\frac{1}{2}}|\nabla c_{\epsilon}|}{1+v_{\epsilon}^{\frac{1}{2}}c_{\epsilon}}\leq \frac{v_{\epsilon}^{\frac{1}{2}}|\nabla c_{\epsilon}|}{1+v_{\epsilon}c_{\epsilon}}\leq \frac{v_{\epsilon}^{\frac{1}{2}}|\nabla c_{\epsilon}|}{(1+v_{\epsilon}c_{\epsilon})^{\frac{1}{2}}}.\label{estgr2}
\end{align}
Combining \eqref{a4} and \eqref{estgr2},
 we obtain that
 \begin{align}
  \left\|\frac{v_{\epsilon}^{\frac{1}{2}}|\nabla c_{\epsilon}|}{1+v_{\epsilon}^{\frac{1}{2}}c_{\epsilon}}\right\|_{L^2((0,T)\times\Omega)}\leq\C(T).\label{estgr2_}
 \end{align}
Altogether, we obtain from \eqref{gru} with \eqref{estgr1_} and \eqref{estgr2_} that
\begin{align}
 ||\nabla u_{\epsilon}||_{L^{\frac{4}{3}}((0,T)\times\Omega)}\leq\C(T).\label{estgru}
\end{align}
Next, we deal with the time derivative of $u_{\epsilon}$. Once again, it holds that
\begin{align}
 \partial_t u_{\epsilon}=\frac{c_{\epsilon}}{1+v_{\epsilon}^{\frac{1}{2}}c_{\epsilon}}\partial_t v_{\epsilon}^{\frac{1}{2}}+\frac{v_{\epsilon}^{\frac{1}{2}}}{1+v_{\epsilon}^{\frac{1}{2}}c_{\epsilon}}\partial_t c_{\epsilon}.\label{tu}
\end{align}
Using \eqref{hopf}, we obtain for the first summand on the right-hand side of \eqref{tu} that
\begin{align}
 \frac{c_{\epsilon}\left|\partial_t v_{\epsilon}^{\frac{1}{2}}\right|}{1+v_{\epsilon}^{\frac{1}{2}}c_{\epsilon}}\leq c_{\epsilon}\left|\partial_t v_{\epsilon}^{\frac{1}{2}}\right|.\label{est1s}
\end{align}
Combining \eqref{est1s} with \eqref{c2} and \eqref{vt}, we obtain that
\begin{align}
 \left\|\frac{c_{\epsilon}\partial_t v_{\epsilon}^{\frac{1}{2}}}{1+v_{\epsilon}^{\frac{1}{2}}c_{\epsilon}}\right\|_{L^1((0,T)\times\Omega)}\leq\C(T).\label{axx0}
\end{align}
In order to estimate the second summand on the right-hand side of \eqref{tu}, we multiply both sides of equation \eqref{eq1e} by $\frac{v_{\epsilon}^{\frac{1}{2}}}{1+v_{\epsilon}^{\frac{1}{2}}c_{\epsilon}}$ and obtain (compare notation, \eqref{qe}-\eqref{fe})  that
\begin{align}
 \frac{v_{\epsilon}^{\frac{1}{2}}}{1+v_{\epsilon}^{\frac{1}{2}}c_{\epsilon}}\partial_t c_{\epsilon}=\nabla\cdot \left(\frac{v_{\epsilon}^{\frac{1}{2}}}{1+v_{\epsilon}^{\frac{1}{2}}c_{\epsilon}}q_{\epsilon}\right)- q_{\epsilon}\cdot\nabla \frac{v_{\epsilon}^{\frac{1}{2}}}{1+v_{\epsilon}^{\frac{1}{2}}c_{\epsilon}} +\frac{v_{\epsilon}^{\frac{1}{2}}}{1+v_{\epsilon}^{\frac{1}{2}}c_{\epsilon}}f_{\epsilon}.\label{ax1}
\end{align}
Since 
\begin{align}
 \frac{v_{\epsilon}^{\frac{1}{2}}}{1+v_{\epsilon}^{\frac{1}{2}}c_{\epsilon}}\leq1,\nonumber
\end{align}
estimates \eqref{f1} and \eqref{qe1_} yield, respectively, that
\begin{align}
&\left\|\frac{v_{\epsilon}^{\frac{1}{2}}}{1+v_{\epsilon}^{\frac{1}{2}}c_{\epsilon}}q_{\epsilon}\right\|_{L^1((0,T)\times\Omega)}\leq\C(T),\label{axx1}\\
 &\left\|\frac{v_{\epsilon}^{\frac{1}{2}}}{1+v_{\epsilon}^{\frac{1}{2}}c_{\epsilon}}f_{\epsilon}\right\|_{L^1((0,T)\times\Omega)}\leq\C(T)\label{axx2}
\end{align}
It remains to estimate the second term on the right-hand side of \eqref{ax1}.
We compute that
\begin{align}
 \nabla\frac{v_{\epsilon}^{\frac{1}{2}}}{1+v_{\epsilon}^{\frac{1}{2}}c_{\epsilon}}=-\frac{v_{\epsilon}}{\left(1+v_{\epsilon}^{\frac{1}{2}}c_{\epsilon}\right)^2}\nabla c_{\epsilon}+\frac{1}{\left(1+v_{\epsilon}^{\frac{1}{2}}c_{\epsilon}\right)^2}\nabla v_{\epsilon}^{\frac{1}{2}},\nonumber
\end{align}
so that, due to \eqref{hopf},
\begin{align}
 \left|q_{\epsilon}\cdot\nabla \frac{v_{\epsilon}^{\frac{1}{2}}}{1+v_{\epsilon}^{\frac{1}{2}}c_{\epsilon}}\right|\leq&|q_{\epsilon}|\left|\nabla \frac{v_{\epsilon}^{\frac{1}{2}}}{1+v_{\epsilon}^{\frac{1}{2}}c_{\epsilon}}\right|\nonumber\\
 \leq&\left(\epsilon_{2}|\nabla c_{\epsilon}|+\frac{\kappa_c v_{\epsilon}c_{\epsilon}|\nabla c_{\epsilon}|}{1+v_{\epsilon}c_{\epsilon}}+\frac{2\kappa_v v_{\epsilon}^{\frac{1}{2}}c_{\epsilon}|\nabla \psi(v_{\epsilon})|}{1+v_{\epsilon}}\right)\left(\frac{v_{\epsilon}|\nabla c_{\epsilon}|}{\left(1+v_{\epsilon}^{\frac{1}{2}}c_{\epsilon}\right)^2}+\frac{\left|\nabla v_{\epsilon}^{\frac{1}{2}}\right|}{\left(1+v_{\epsilon}^{\frac{1}{2}}c_{\epsilon}\right)^2}\right)
 \nonumber\\
 \leq&\C\left(\epsilon_{2}|\nabla c_{\epsilon}|+\frac{ v_{\epsilon}c_{\epsilon}|\nabla c_{\epsilon}|}{1+v_{\epsilon}c_{\epsilon}}+v_{\epsilon}^{\frac{1}{2}}c_{\epsilon}\left|\nabla v_{\epsilon}^{\frac{1}{2}}\right|\right)\left(\frac{v_{\epsilon}|\nabla c_{\epsilon}|}{\left(1+v_{\epsilon}^{\frac{1}{2}}c_{\epsilon}\right)^2}+\frac{\left|\nabla v_{\epsilon}^{\frac{1}{2}}\right|}{\left(1+v_{\epsilon}^{\frac{1}{2}}c_{\epsilon}\right)^2}\right).\label{ax2}
\end{align}
Using \eqref{hopf} and \eqref{triv}, where necessary, we get the following estimates:
\begin{align}
& |\nabla c_{\epsilon}|\frac{v_{\epsilon}|\nabla c_{\epsilon}|}{\left(1+v_{\epsilon}^{\frac{1}{2}}c_{\epsilon}\right)^2}\leq \frac{v_{\epsilon}|\nabla c_{\epsilon}|^2}{1+v_{\epsilon}c_{\epsilon}},\label{ax3}
\end{align}
\begin{align}
 v_{\epsilon}^{\frac{1}{2}}c_{\epsilon}\left|\nabla v_{\epsilon}^{\frac{1}{2}}\right|\frac{v_{\epsilon}|\nabla c_{\epsilon}|}{\left(1+v_{\epsilon}^{\frac{1}{2}}c_{\epsilon}\right)^2}\leq &\left(\frac{v_{\epsilon}|\nabla c_{\epsilon}|^2}{1+v_{\epsilon}c_{\epsilon}}\right)^{\frac{1}{2}}\left|\nabla v_{\epsilon}^{\frac{1}{2}}\right|\nonumber\\
\leq &\frac{1}{2}\frac{v_{\epsilon}|\nabla c_{\epsilon}|^2}{1+v_{\epsilon}c_{\epsilon}}+\frac{1}{2}\left|\nabla v_{\epsilon}^{\frac{1}{2}}\right|^2,
\end{align}
\begin{align}
 |\nabla c_{
 \epsilon}|\frac{\left|\nabla v_{\epsilon}^{\frac{1}{2}}\right|}{\left(1+v_{\epsilon}^{\frac{1}{2}}c_{\epsilon}\right)^2}\leq&\frac{v_{\epsilon}^{\frac{1}{2}}|\nabla c_{\epsilon}|\left|\nabla v_{\epsilon}^{\frac{1}{2}}\right|}{\left(1+v_{\epsilon}c_{\epsilon}\right)^2}+\frac{(1-v_{\epsilon})^{\frac{1}{2}}|\nabla c_{\epsilon}|\left|\nabla v_{\epsilon}^{\frac{1}{2}}\right|}{\left(1+v_{\epsilon}c_{\epsilon}\right)^2}\nonumber\\
 \leq &\left(\frac{v_{\epsilon}|\nabla c_{\epsilon}|^2}{1+v_{\epsilon}c_{\epsilon}}\right)^{\frac{1}{2}}\left|\nabla v_{\epsilon}^{\frac{1}{2}}\right|+\left(\frac{|\nabla c_{\epsilon}|^2}{c_{\epsilon}}\right)^{\frac{1}{2}}\left((1-v_{\epsilon})c_{\epsilon}\left|\nabla v_{\epsilon}^{\frac{1}{2}}\right|^2\right)^{\frac{1}{2}}\nonumber\\
 \leq&\frac{1}{2}\frac{v_{\epsilon}|\nabla c_{\epsilon}|^2}{1+v_{\epsilon}c_{\epsilon}}+\frac{1}{2}\left|\nabla v_{\epsilon}^{\frac{1}{2}}\right|^2+\frac{1}{2}(1-v_{\epsilon})c_{\epsilon}\left|\nabla v_{\epsilon}^{\frac{1}{2}}\right|^2+\frac{1}{2}\frac{|\nabla c_{\epsilon}|^2}{c_{\epsilon}},
\end{align}
\begin{align}
 \frac{v_{\epsilon}c_{\epsilon}|\nabla c_{\epsilon}|}{1+v_{\epsilon}c_{\epsilon}}\frac{\left|\nabla v_{\epsilon}^{\frac{1}{2}}\right|}{\left(1+v_{\epsilon}^{\frac{1}{2}}c_{\epsilon}\right)^2}\leq &\left(\frac{v_{\epsilon}|\nabla c_{\epsilon}|^2}{1+v_{\epsilon}c_{\epsilon}}\right)^{\frac{1}{2}}\left|\nabla v_{\epsilon}^{\frac{1}{2}}\right|\nonumber\\
\leq &\frac{1}{2}\frac{v_{\epsilon}|\nabla c_{\epsilon}|^2}{1+v_{\epsilon}c_{\epsilon}}+\frac{1}{2}\left|\nabla v_{\epsilon}^{\frac{1}{2}}\right|^2,
\end{align}
\begin{align}
 v_{\epsilon}^{\frac{1}{2}}c_{\epsilon}\left|\nabla v_{\epsilon}^{\frac{1}{2}}\right|\frac{\left|\nabla v_{\epsilon}^{\frac{1}{2}}\right|}{\left(1+v_{\epsilon}^{\frac{1}{2}}c_{\epsilon}\right)^2}\leq &\left|\nabla v_{\epsilon}^{\frac{1}{2}}\right|^2.\label{ax4}
\end{align}
Combining \eqref{ax2}-\eqref{ax4} with \eqref{a2}, \eqref{a4}-\eqref{a6}, we obtain that
\begin{align}
 \left\|q_{\epsilon}\cdot\nabla \frac{v_{\epsilon}^{\frac{1}{2}}}{1+v_{\epsilon}^{\frac{1}{2}}c_{\epsilon}}\right\|_{L^1((0,T)\times\Omega)}\leq\C(T).\label{ax5}
\end{align}
Therefore,  \eqref{ax1}-\eqref{axx2} together with \eqref{ax5} yield that
\begin{align}
 \left\|\frac{v_{\epsilon}^{\frac{1}{2}}\partial_t c_{\epsilon}}{1+v_{\epsilon}^{\frac{1}{2}}c_{\epsilon}}\right\|_{L^1(0,T;W^{-1,1}(\Omega))}\leq\C(T).\label{ax6}
\end{align}
Finally, with help of estimates \eqref{axx0} and \eqref{ax6}, we obtain from \eqref{tu} that
\begin{align}
||\partial_t u_{\epsilon}||_{L^1(0,T;W^{-1,1}(\Omega))}\leq \C(T).\label{uet}
\end{align}
\section{Global existence for the original problem}\label{existence}
In this section we aim to pass to the limit in \eqref{chemoe} in order to obtain a solution of the original problem. 
\begin{Remark}[Notation]
 Let $ \{\epsilon_{i,n_{i}}\}\subset (0,1)$, $i= 1,2,3,4$, be four sequences such that for each $i= 1,2,3,4$ it holds that $\epsilon_{i,n_{i}}
\underset{n_{i}\rightarrow\infty}{\rightarrow}0$. In this section, we make use of the following vector notation:
\begin{align}
 &n_{i:4}:=\left(n_{i},\dots,n_{4}\right), \
\epsilon_{n_{i:4}}:=\left(\epsilon_{i,n_{i}},\dots,\epsilon_{4,n_4}\right),\ i=1,2,3.\nonumber
\end{align}
Let us illustrate the way we are going to apply it. Let $a_{\epsilon}$ be a family parameterized by $\epsilon$, i.e., by the quadruple $\left(\epsilon_{1},\epsilon_{2},\epsilon_{3},\epsilon_{4}\right)$. By writing
\begin{align}
 a_{\epsilon_{n_{1:4}}}\underset{n_{1}\rightarrow\infty}{\rightarrow}a_{n_{2:4}}\underset{n_{2:4}\rightarrow\infty}{\rightarrow} a\text{ (in some topology)},\label{limit}
\end{align}
we mean thus that the sequence $\{a_{\epsilon_{n_{1:4}}}\}$ converges to some  $a_{n_{2:4}}$ as $n_{1}\rightarrow\infty$ for each $n_{2:4}$, while $\{a_{n_{2:4}}\}$ converges 
to some $a$ (in some topology) as $n_{2:4}\rightarrow\infty$, i.e., as $n_{2},n_{3},n_{4}\rightarrow\infty$. 
As for the family of limits $\{a_{n_{2:4}}\}$ and the limit $a$, it is assumed that they  either have been previously introduced, or that they exist and are  being thus 
introduced by expression \eqref{limit}. \\
\noindent
Thereby, we can write two subsequent limit procedures in a compact form.

\end{Remark}

\noindent
From know on, we assume that the families of initial values $\{c_{\epsilon_{3}0}\},\{v_{\epsilon_{4}0}\}$  are  independent of $\epsilon_{1}$ and $\epsilon_{2}$ and satisfy \eqref{e34bnd}-\eqref{star}.  Recall that (compare \eqref{eqd}-\eqref{fe}) \eqref{eq1e} can be rewritten in the following form:
\begin{align}
 \partial_t c_{\epsilon}=\nabla\cdot q_{\epsilon}+f_{\epsilon},\label{eqd1}
\end{align}
where 
\begin{align}
&q_{\epsilon}:=\epsilon_{2}\nabla c_{\epsilon}+\frac{\kappa_c v_{\epsilon}c_{\epsilon}}{1+v_{\epsilon}c_{\epsilon}}\nabla c_{\epsilon}-\frac{2\kappa_v v_{\epsilon}^{\frac{1}{2}}c_{\epsilon}}{(1+v_{\epsilon})^2}\nabla v_{\epsilon}^{\frac{1}{2}},\nonumber\\
& f_{\epsilon}:=\mu_c c_{\epsilon} (1-c_{\epsilon}-\eta v_{\epsilon})-\epsilon_{1} c_{\epsilon}^{\theta}\nonumber
\end{align}
are the flux vector and the reaction term, respectively.

Owing to the estimates obtained in the preceding section,
there exist four sequences $$ \epsilon_{i,n_{i}}\underset{n_{i}\rightarrow\infty}{\rightarrow}0,\ i= 1,2,3,4,$$
such that:\\
due to \eqref{c2} and the Banach-Alaoglu theorem 
\begin{align}
&c_{\epsilon_{n_{1:4}}}\underset{n_{1}\rightarrow\infty}{\rightharpoonup}c_{n_{2:4}}\underset{n_{2:4}\rightarrow\infty}{\rightharpoonup}c\text{ in }L^{2}((0,T)\times\Omega);\label{cconv2}
\end{align}
due to \eqref{ui} and the Dunford-Pettis theorem
\begin{align}
&c_{\epsilon_{n_{1:4}}}^2\underset{n_{1}\rightarrow\infty}{\rightharpoonup}\tilde{c}_{n_{2:4}}^2\underset{n_{2:4}\rightarrow\infty}{\rightharpoonup}\tilde{c}^2\text{ in }L^{1}((0,T)\times\Omega);\label{c2conv}
\end{align}
due to \eqref{a2}, \eqref{vt} and the Lions-Aubin lemma
\begin{align}
  v_{\epsilon_{n_{1:4}}}^{\frac{1}{2}}\underset{n_{1}\rightarrow\infty}{\rightarrow}v^{\frac{1}{2}}_{n_{2:4}}\underset{n_{2:4}\rightarrow\infty}{\rightarrow} v^{\frac{1}{2}}\text{ in }L^{2}((0,T)\times\Omega);\label{vcomp}
\end{align}
due to \eqref{vcomp}
\begin{align}
 v_{\epsilon_{n_{1:4}}}^{\frac{1}{2}}\underset{n_{1}\rightarrow\infty}{\rightarrow}v^{\frac{1}{2}}_{n_{2:4}}\underset{n_{2:4}\rightarrow\infty}{\rightarrow} v^{\frac{1}{2}}\text{ a.e. in  }(0,T)\times\Omega;\label{vae}
\end{align}
due to \eqref{vae}
\begin{align}
 v_{\epsilon_{n_{1:4}}}\underset{n_{1}\rightarrow\infty}{\rightarrow}v_{n_{2:4}}\underset{n_{2:4}\rightarrow\infty}{\rightarrow} v\text{ a.e. in  }(0,T)\times\Omega;\label{vae_}
\end{align}
due to \eqref{hopf}, \eqref{vae_} and the  dominated convergence theorem
\begin{align}
 v_{\epsilon_{n_{1:4}}}^a\underset{n_{1}\rightarrow\infty}{\rightarrow}v_{n_{2:4}}^a\underset{n_{2:4}\rightarrow\infty}{\rightarrow} v^a\text{ in }L^{p}((0,T)\times\Omega)\text{ for all }a>0,\ p\geq1;\label{vconvp}
\end{align}
due to \eqref{cconv2}, \eqref{c2conv} and \eqref{vconvp}
\begin{eqnarray}
 f_{\epsilon_{n_{1:4}}}&=&\mu_c c_{\epsilon_{n_{1:4}}} \left(1-c_{\epsilon_{n_{1:4}}}-\eta v_{\epsilon_{n_{1:4}}}\right)-\epsilon_{1,n_{1}} c_{\epsilon_{n_{1:4}}}^{\theta}\nonumber\\
 &\underset{n_{1}\rightarrow\infty}{\rightharpoonup}&\mu_cc_{n_{2:4}}(1-\eta v_{n_{2:4}})-\mu_c \tilde{c}^2_{n_{2:4}}=:f_{n_{2:4}}\nonumber\\
 &\underset{n_{2:4}\rightarrow\infty}{\rightharpoonup}&\mu_cc(1-\eta v)-\mu_c \tilde{c}^2=:f\text{ in }L^{1}((0,T)\times\Omega);\label{fconv}
 \end{eqnarray}
due to \eqref{a5}, \eqref{vcomp} and the Banach-Alaoglu theorem
\begin{align}
 \nabla v_{\epsilon_{n_{1:4}}}^{\frac{1}{2}}\underset{n_{1}\rightarrow\infty}{\rightharpoonup}\nabla v_{n_{2:4}}^{\frac{1}{2}}\underset{n_{2:4}\rightarrow\infty}{\rightharpoonup}\nabla v^{\frac{1}{2}}\text{ in }L^{2}((0,T)\times\Omega);\label{nv2}
\end{align}
due to \eqref{nce}, \eqref{ct1} and a version of the Lions-Aubin Lemma \cite[Corollary 4]{Simon}
\begin{align}
 c_{\epsilon_{n_{1:4}}}\underset{n_{1}\rightarrow\infty}{\rightarrow}c_{n_{2:4}}\text{ in }L^{\frac43}((0,T)\times\Omega);\label{compcm}
\end{align}
due to \eqref{compcm}
\begin{align}
 c_{\epsilon_{n_{1:4}}}\underset{n_{1}\rightarrow\infty}{\rightarrow}c_{n_{2:4}}\text{ a.e. in  }(0,T)\times\Omega;\label{ce2ae}
\end{align}
due to \eqref{cconv2}-\eqref{c2conv} and  \eqref{compcm}
\begin{align}
 c_{\epsilon_{n_{1:4}}}\underset{n_{1}\rightarrow\infty}{\rightarrow}c_{n_{2:4}}=\tilde{c}_{n_{2:4}}\text{ a.e. in  }(0,T)\times\Omega;\label{cn}
\end{align}
due to \eqref{nce}, \eqref{compcm} and the Banach-Alaoglu theorem
\begin{align}
 \nabla c_{\epsilon_{n_{1:4}}}\underset{n_{1}\rightarrow\infty}{\rightharpoonup}\nabla c_{n_{2:4}}\text{ in }L^{\frac{4}{3}}((0,T)\times\Omega);\label{ncconv}
\end{align}
due to \eqref{estgru}, \eqref{uet} and a version of the Lions-Aubin Lemma \cite[Corollary 4]{Simon}
\begin{align}
 \ln\left(1+v_{n_{2:4}}^{\frac{1}{2}}c_{n_{2:4}}\right)\underset{n_{2:4}\rightarrow\infty}{\rightarrow} u\text{ in }L^{\frac{4}{3}}((0,T)\times\Omega);\label{ucomp2}
\end{align}
due to \eqref{ucomp2}
\begin{align}
 \ln\left(1+v_{n_{2:4}}^{\frac{1}{2}}c_{n_{2:4}}\right)\underset{n_{2:4}\rightarrow\infty}{\rightarrow} u\text{ a.e. in  }(0,T)\times\Omega;\label{uae}
\end{align}
due to \eqref{uae},
\begin{align}
 v_{n_{2:4}}^{\frac{1}{2}}c_{n_{2:4}}\underset{n_{2:4}\rightarrow\infty}{\rightarrow} e^u-1=:w\text{ a.e. in  }(0,T)\times\Omega;\label{wae}
\end{align}
due to \eqref{cconv2}-\eqref{c2conv}, \eqref{vae_}, \eqref{wae} and the Lions lemma \cite[Lemma 1.3]{Lions}
\begin{align}
 c_{n_{2:4}}\underset{n_{2:4}\rightarrow\infty}{\rightarrow}c=\tilde{c}=\frac{w}{v^{\frac{1}{2}}}\text{ a.e. in  }\{v>0\};\label{ae}
\end{align}
due to \eqref{wae}-\eqref{ae}
\begin{align}
 v_{n_{2:4}}^{\frac{1}{2}}c_{n_{2:4}}\underset{n_{2:4}\rightarrow\infty}{\rightarrow}v^{\frac{1}{2}}c\text{ a.e. in  }(0,T)\times\Omega;\label{w}
\end{align}
due to \eqref{wui1}, \eqref{w} and the Vitali convergence theorem 
\begin{align}
 v_{n_{2:4}}^{\frac{1}{2}}c_{n_{2:4}}\underset{n_{2:4}\rightarrow\infty}{\rightarrow}v^{\frac{1}{2}}c\text{ in }L^{2}((0,T)\times\Omega);\label{ta1}
\end{align}
due to \eqref{ta1} and $w\mapsto\ln(1+w)$ being a Lipschitz function in $\R_0^+$
\begin{align}
 \ln\left(1+v_{n_{2:4}}^{\frac{1}{2}}c_{n_{2:4}}\right)\underset{n_{2:4}\rightarrow\infty}{\rightarrow}\ln\left(1+v^{\frac{1}{2}}c\right)\text{ in }L^{2}((0,T)\times\Omega);\label{ta1_}
\end{align}
due to \eqref{estgru}, \eqref{ta1_} and the Banach-Alaoglu theorem
\begin{align}
 \nabla \ln\left(1+v_{n_{2:4}}^{\frac{1}{2}}c_{n_{2:4}}\right)\underset{n_{2:4}\rightarrow\infty}{\rightharpoonup}\nabla \ln\left(1+v^{\frac{1}{2}}c\right)\text{ in }L^{\frac{4}{3}}((0,T)\times\Omega);\label{nu43}
\end{align}
due to \eqref{vae}, \eqref{cn} and \eqref{w},
\begin{align}
 \frac{2v_{\epsilon_{n_{1:4}}}^{\frac{1}{2}}c_{\epsilon_{n_{1:4}}}}{\left(1+v_{\epsilon_{n_{1:4}}}\right)^2}\underset{n_{1}\rightarrow\infty}{\rightarrow}\frac{2 v_{n_{2:4}}^{\frac{1}{2}}c_{n_{2:4}}}{(1+v_{n_{2:4}})^2}\underset{n_{2:4}\rightarrow\infty}{\rightarrow}\frac{2 v^{\frac{1}{2}}c}{(1+v)^2}\text{ a.e. in }(0,T)\times\Omega;\label{ta4}
\end{align}
due to \eqref{wui1}, \eqref{ta4}, $\frac{1}{(1+v)^2}\leq 1$ for $v\in \R^+_0$ and the Vitali convergence theorem 
\begin{align}
 \frac{2v_{\epsilon_{n_{1:4}}}^{\frac{1}{2}}c_{\epsilon_{n_{1:4}}}}{\left(1+v_{\epsilon_{n_{1:4}}}\right)^2}\underset{n_{1}\rightarrow\infty}{\rightarrow}\frac{2 v_{n_{2:4}}^{\frac{1}{2}}c_{n_{2:4}}}{(1+v_{n_{2:4}})^2}\underset{n_{2:4}\rightarrow\infty}{\rightarrow}\frac{2 v^{\frac{1}{2}}c}{(1+v)^2}\text{ in }L^{2}((0,T)\times\Omega);\label{ta5}
\end{align}
due to \eqref{nv2}, \eqref{ta5} and the well-known result on weak-strong convergence for member-by-member products
\begin{align}
 \frac{2  v_{\epsilon_{n_{1:4}}}^{\frac{1}{2}}c_{\epsilon_{n_{1:4}}}}{\left(1+v_{\epsilon_{n_{1:4}}}\right)^2}\nabla v_{\epsilon_{n_{1:4}}}^{\frac{1}{2}}\underset{n_{1}\rightarrow\infty}{\rightharpoonup}\frac{2 v_{n_{2:4}}^{\frac{1}{2}}c_{n_{2:4}}}{(1+v_{n_{2:4}})^2}\nabla v^{\frac{1}{2}}_{n_{2:4}}\underset{n_{2:4}\rightarrow\infty}{\rightharpoonup}&\frac{2 v^{\frac{1}{2}}c}{(1+v)^2}\nabla v^{\frac{1}{2}}
\text{ in }L^1((0,T)\times\Omega);\label{tcomp}
\end{align}
similarly, due to \eqref{vae_}, \eqref{cn}-\eqref{ncconv}, $\frac{vc}{1+vc}\leq1$, the dominated convergence theorem and the result on weak-strong convergence for member-by-member products
\begin{align}
 \frac{v_{\epsilon_{n_{1:4}}}c_{\epsilon_{n_{1:4}}}}{1+v_{\epsilon_{n_{1:4}}}c_{\epsilon_{n_{1:4}}}}\nabla c_{\epsilon_{n_{1:4}}}\underset{n_{1}\rightarrow\infty}{\rightharpoonup}&\frac{v_{n_{2:4}}c_{n_{2:4}}}{1+v_{n_{2:4}}c_{n_{2:4}}}\nabla c_{n_{2:4}}\text{ in }L^{\frac{4}{3}}((0,T)\times\Omega);\label{dcompn}
\end{align}
due to \eqref{dif2}
\begin{align}
 \frac{v_{n_{2:4}}c_{n_{2:4}}}{1+v_{n_{2:4}}c_{n_{2:4}}}\nabla c_{n_{2:4}}\underset{n_{2:4}\rightarrow\infty}{\rightharpoonup}&d\text{ in }L^{\frac{4}{3}}((0,T)\times\Omega);\label{dcomp}
\end{align}
due to \eqref{dif2_}
\begin{align}
 \frac{v_{n_{2:4}}^{\frac{3}{4}}c_{n_{2:4}}}{1+v_{n_{2:4}}c_{n_{2:4}}}\nabla c_{n_{2:4}}\underset{n_{2:4}\rightarrow\infty}{\rightharpoonup}&d_1\text{ in }L^{\frac{8}{7}}((0,T)\times\Omega);\label{dcomp_}
\end{align}
due to \eqref{vconvp}, \eqref{dcomp_} and the  result on weak-strong convergence for member-by-member products
\begin{eqnarray}
 \frac{v_{n_{2:4}}c_{n_{2:4}}}{1+v_{n_{2:4}}c_{n_{2:4}}}\nabla c_{n_{2:4}}&=&v_{n_{2:4}}^{\frac{1}{4}}\frac{v_{n_{2:4}}^{\frac{3}{4}}c_{n_{2:4}}}{1+v_{n_{2:4}}c_{n_{2:4}}}\nabla c_{n_{2:4}}\nonumber\\
 &\underset{n_{2:4}\rightarrow\infty}{\rightharpoonup}&0\text{ in }L^{1}(((0,T)\times\Omega)\cap\{v=0\});\label{dcomp0}
 \end{eqnarray}
due to \eqref{vae_}, \eqref{nv2}, \eqref{ae},  \eqref{dcomp} and \eqref{nu43} and {\it Lemma~\ref{LemA2}} from Appendix 
\begin{align}
 \frac{v_{n_{2:4}}c_{n_{2:4}}}{1+v_{n_{2:4}}c_{n_{2:4}}}\nabla c_{n_{2:4}}
 =&\frac{v_{n_{2:4}}^{\frac{1}{2}}c_{n_{2:4}}}{1+v_{n_{2:4}}c_{n_{2:4}}}\nabla\left( v_{n_{2:4}}^{\frac{1}{2}} c_{n_{2:4}}\right)-\frac{v_{n_{2:4}}^{\frac{1}{2}}c_{n_{2:4}}}{1+v_{n_{2:4}}c_{n_{2:4}}}c_{n_{2:4}}\nabla v_{n_{2:4}}^{\frac{1}{2}}\nonumber\\
 =&\frac{v_{n_{2:4}}^{\frac{1}{2}}c_{n_{2:4}}\left(1+v_{n_{2:4}}^{\frac{1}{2}} c_{n_{2:4}}\right)}{1+v_{n_{2:4}}c_{n_{2:4}}}\nabla  \ln\left(1+v_{\epsilon}^{\frac{1}{2}}c_{\epsilon}\right)-\frac{v_{n_{2:4}}^{\frac{1}{2}}c_{n_{2:4}}^2}{1+v_{n_{2:4}}c_{n_{2:4}}}\nabla v_{n_{2:4}}^{\frac{1}{2}}\nonumber\\
 \underset{n_{2:4}\rightarrow\infty}{\rightharpoonup}&\frac{v^{\frac{1}{2}}c\left(1+v^{\frac{1}{2}}c\right)}{1+vc}\nabla \ln\left(1+v^{\frac{1}{2}}c\right)-\frac{v^{\frac{1}{2}}c^2}{1+vc}\nabla v^{\frac{1}{2}}\nonumber\\
 =&\frac{v^{\frac{1}{2}}c}{1+vc}\left(\left(1+v^{\frac{1}{2}}c\right)\nabla \ln\left(1+v^{\frac{1}{2}}c\right)-c\nabla v^{\frac{1}{2}}\right)\text{ in }L^1(((0,T)\times\Omega)\cap\{v>0\});\label{dcompn0}
 \end{align}
due to \eqref{dcomp0}-\eqref{dcompn0}
\begin{align}
 \frac{v_{n_{2:4}}c_{n_{2:4}}}{1+v_{n_{2:4}}c_{n_{2:4}}}\nabla c_{n_{2:4}}\underset{n_{2:4}\rightarrow\infty}{\rightharpoonup}&\frac{v^{\frac{1}{2}}c}{1+vc}\left(\left(1+v^{\frac{1}{2}}c\right)\nabla \ln\left(1+v^{\frac{1}{2}}c\right)-c\nabla v^{\frac{1}{2}}\right)\text{ in }L^1((0,T)\times\Omega);\label{d_comp}
\end{align}
due to \eqref{este2}, \eqref{tcomp}-\eqref{dcompn} and \eqref{d_comp}
\begin{eqnarray}
 q_{\epsilon_{n_{1:4}}}&=&\epsilon_{2,n_{2}}\nabla c_{\epsilon_{n_{1:4}}}+\frac{\kappa_c v_{\epsilon_{n_{1:4}}}c_{\epsilon_{n_{1:4}}}}{1+v_{\epsilon_{n_{1:4}}}c_{\epsilon_{n_{1:4}}}}\nabla c_{\epsilon_{n_{1:4}}}-\frac{2\kappa_v v_{\epsilon_{n_{1:4}}}^{\frac{1}{2}}c_{\epsilon_{n_{1:4}}}}{\left(1+v_{\epsilon_{n_{1:4}}}\right)^2}\nabla v_{\epsilon_{n_{1:4}}}^{\frac{1}{2}}\nonumber\\
 &\underset{n_{1}\rightarrow\infty}{\rightharpoonup}&\epsilon_{2,n_{2}}\nabla c_{n_{2:4}}+\frac{\kappa_c v_{n_{2:4}}c_{n_{2:4}}}{1+v_{n_{2:4}}c_{n_{2:4}}}\nabla c_{n_{2:4}}-\frac{2\kappa_v v_{n_{2:4}}^{\frac{1}{2}}c_{n_{2:4}}}{(1+v_{n_{2:4}})^2}\nabla v_{n_{2:4}}^{\frac{1}{2}}=:q_{n_{2:4}}\nonumber\\
 &\underset{n_{2:4}\rightarrow\infty}{\rightharpoonup}&\frac{v^{\frac{1}{2}}c}{1+vc}\left(\left(1+v^{\frac{1}{2}}c\right)\nabla \ln\left(1+v^{\frac{1}{2}}c\right)-c\nabla v^{\frac{1}{2}}\right)-\frac{2 v^{\frac{1}{2}}c}{(1+v)^2}\nabla v^{\frac{1}{2}}\nonumber\\
 &=:&q\text{ in }L^{1}((0,T)\times\Omega);\label{qcomp}
\end{eqnarray}
due to \eqref{eqd1}, \eqref{cconv2}, \eqref{fconv} and \eqref{qcomp} 
\begin{align}
  &\left<\partial_t c_{n_{2:4}},\varphi\right>=-(q_{n_{2:4}},\nabla\varphi)+(f_{n_{2:4}},\varphi)\text{ a.e.  in }\R^+\text{ for all }\varphi\in W^{1,\infty}(\Omega),\label{eqlimn}\\
 &\left<\partial_t c,\varphi\right>=-(q,\nabla\varphi)+(f,\varphi)\text{ a.e.  in }\R^+\text{ for all }\varphi\in W^{1,\infty}(\Omega),\label{eqlim}
\end{align} 
and the limiting identities \eqref{eqlimn}-\eqref{eqlim} have in $L^1(0,T;W^{-1,1}(\Omega))$ the form  
\begin{align}
\partial_t c_{n_{2:4}}=&\epsilon_{2,n_{2}}\Delta c_{n_{2:4}}+\nabla\cdot\left(\frac{\kappa_c v_{n_{2:4}}c_{n_{2:4}}}{1+v_{n_{2:4}}c_{n_{2:4}}}\nabla c_{n_{2:4}}-\frac{\kappa_v c_{n_{2:4}}}{(1+v_{n_{2:4}})^2}\nabla v_{n_{2:4}}\right)+\mu_cc_{n_{2:4}}(1-\eta v_{n_{2:4}})-\mu_c c_{n_{2:4}}^2,\label{limeqn}\\
\partial_t c=&\nabla\cdot\left(\frac{\kappa_c v^{\frac{1}{2}}c}{1+vc}\left(\left(1+v^{\frac{1}{2}}c\right)\nabla \ln\left(1+v^{\frac{1}{2}}c\right)-c\nabla v^{\frac{1}{2}}\right)-\frac{\kappa_v c}{(1+v)^2}\nabla v\right)+\mu_cc(1-\eta v)-\mu_c \tilde{c}^2.\label{limeq}
\end{align}
Further, using \eqref{a8}, \eqref{cconv2}, \eqref{vconvp}, and the fact that equations \eqref{v} and \eqref{v2} are equivalent, we obtain that
\begin{align}
 &\partial_t v_{n_{2:4}}=\mu_vv_{n_{2:4}}(1-v_{n_{2:4}})-\lambda v_{n_{2:4}}c_{n_{2:4}}\text{ a.e. in }\R^+\times\Omega,\label{vn}\\
 &\partial_t v=\mu_vv(1-v)-\lambda vc\text{ a.e. in }\R^+\times\Omega.\label{v_}
\end{align}
For the initial data, we have with \eqref{dpsi}, and \eqref{coe}-\eqref{voe} that
\begin{align}
 &c_{n_{3}0}:=c_{n_{2:4}}(0)=c_{\epsilon_{n_{1:4}0}}\underset{n_{3}\rightarrow\infty}{\rightarrow}c_0\text{ in }L^1(\Omega)\text{ and a.e. in }\Omega,\label{ck0}\\
 &v_{n_{4}0}^{\frac{1}{2}}:=v_{n_{2:4}}^{\frac{1}{2}}(0)=v_{\epsilon_{n_{1:4}}0}^{\frac{1}{2}}\underset{n_{4}\rightarrow\infty}{\rightarrow}v^{\frac{1}{2}}_0\text{ in }\LTwo.\nonumber
\end{align}
\subsection*{Passing to the limit on \texorpdfstring{$\{v=0\}$}{v=0}:  $c=\tilde{c}$}
With equation \eqref{limeq} we have nearly regained \eqref{c}. However, we still have to check that $c$ and $\tilde{c}$ coincide  a.e. Thanks to \eqref{ae}, it remains to justify that $c=\tilde{c}$ a.e. in  $\{v=0\}$. Observe that this is not obvious since $c$ and $\tilde{c}^2$ are just weak limits of $c_{n_{2:4}}$ and $c_{n_{2:4}}^2$, respectively. 

\noindent
Let us first prove that each level set $\{v_{n_{2:4}}=0\}$ differs from the cylinder $\R_0^+\times \{v_{n_{4}0}=0\}$ by a null set. Indeed, let us divide both sides of the ODE \eqref{vn} by $v_{n_{2:4}}$ and integrate over $(0,t)$ for arbitrary $t\in\R^+$. We obtain that
\begin{align}
 \ln(v_{n_{2:4}}(t))-\ln(v_{n_{4}0}))=&\int_{0}^t\mu_v(1-v_{n_{2:4}})\,dt-\lambda \int_{0}^tc_{n_{2:4}}\,dt.\label{lnv}
\end{align}
Since $0\leq v_{n_{2:4}}\leq1$ and $c_{n_{2:4}}\in L^1((0,T)\times\Omega)$ for all $T\in\R^+$, the right-hand side of \eqref{lnv} is finite a.e. in $\Omega$. Hence, the same holds for the left-hand side of \eqref{lnv}. But this means that for all $t\in\R^+$ it holds that 
\begin{align}
   &v_{n_{2:4}}(t)>0 \text{ a.e. in }\{v_{n_{4}0}>0\},\nonumber\\
   &v_{n_{2:4}}(t)=0 \text{ a.e. in }\{v_{n_{4}0}=0\}.\nonumber
\end{align}
Similarly, we obtain from \eqref{v_} that
\begin{align}
   &v(t)>0 \text{ a.e. in }\{v_0>0\},\label{vpos}\\
   &v(t)=0 \text{ a.e. in }\{v_0=0\}.\nonumber
\end{align}

\noindent
Combining \eqref{ae} and \eqref{vpos}, we conclude that
\begin{align}
 c=\tilde{c}\text{ a.e. in }\R^+\times\{v_0>0\}.
\end{align}
It thus remains to consider $c$ and $\tilde{c}$ in the cylinder $\R^+\times\{v_0=0\}$ for the case when  $$|\{v_0=0\}|\neq 0.$$
We conclude from \eqref{limeqn} and \eqref{ck0} that $c_{n_{2:4}}$ solves
\begin{subequations}\label{lim10}
\begin{alignat}{3}
& \partial_t c_{n_{2:4}}=\epsilon_{2,n_{2}}\Delta c_{n_{2:4}}+\mu_cc_{n_{2:4}}-\mu_c c_{n_{2:4}}^2&&\text{ in }\R^+\times \inner\{v_{n_{4}0}=0\},\\
&c_{n_{2:4}}(0)=c_{n_{3}0}&&\text{ in }\inner\{v_{n_{4}0}=0\}.
\end{alignat}
\end{subequations}
Since $c_{n_{3}0}$ is smooth, $c_{n_{2:4}}$ is a classical solution to \eqref{lim10}. Differentiating \eqref{lim10} with respect to $x_i$, $i\in1:N$, we obtain that
\begin{align}
& \partial_t \partial_{x_i} c_{n_{2:4}}=\epsilon_{2,n_{2}}\Delta\partial_{x_i}  c_{n_{2:4}}+\mu_c(1-c_{n_{2:4}})\partial_{x_i} c_{n_{2:4}}.\label{nlim10}
\end{align}
Let now $\varphi$ be some smooth cut-off function with $\supp\varphi\subset \inner\{v_{n_{4}0}=0\}$. Multiplying \eqref{nlim10} by $\frac{4}{3}\varphi^2|\partial_{x_i} c_{n_{2:4}}|^{-\frac{2}{3}}\partial_{x_i} c_{n_{2:4}}$ and integrating by parts over $\Omega$, we obtain by the H\"older and Young inequalities that
\begin{align}
 \frac{d}{dt}\left\|\varphi|\partial_{x_i} c_{n_{2:4}}|^{\frac{2}{3}}\right\|^2=&-\epsilon_{2,n_{2}}\left\|\varphi\nabla|\partial_{x_i} c_{n_{2:4}}|^{\frac{2}{3}}\right\|^2-4\epsilon_{2,n_{2}}\left(\varphi\nabla|\partial_{x_i} c_{n_{2:4}}|^{\frac{2}{3}},|\partial_{x_i} c_{n_{2:4}}|^{\frac{2}{3}}\nabla\varphi\right)\nonumber\\
 &+\frac{4}{3}\left(\mu_c(1- c_{n_{2:4}}),\varphi^2|\partial_{x_i} c_{n_{2:4}}|^{\frac{4}{3}}\right)\nonumber\\
 \leq&\frac{4}{3}\mu_c\left\|\varphi|\partial_{x_i} c_{n_{2:4}}|^{\frac{2}{3}}\right\|^2+\C||\nabla\varphi||_{\infty}\epsilon_{2,n_{2}}\left\|\partial_{x_i} c_{n_{2:4}}\right\|^{\frac{4}{3}}_{\frac{4}{3}}.\label{43}
\end{align}
Together with \eqref{nce} and the Gronwall lemma,  \eqref{43} yields that 
\begin{align}
 \left\|\varphi|\partial_{x_i} c_{n_{2:4}}|^{\frac{2}{3}}\right\|^2\leq&\Cl{C50}(T)\left\|\varphi|\partial_{x_i} c_{n_{3}0}|^{\frac{2}{3}}\right\|^2+\Cr{C50}(T)||\nabla\varphi||_{\infty}\epsilon_{2,n_{2}}\int_0^T\left\|\partial_{x_i} c_{n_{2:4}}\right\|^{\frac{4}{3}}_{\frac{4}{3}}\,dt\nonumber\\
 \leq&\C\left(T,||\nabla\varphi||_{\infty},\left\|\partial_{x_i} c_{n_{3}0}\right\|_{L^{\frac{4}{3}}(\supp\varphi)}\right).\nonumber
\end{align}
Therefore, for all compacts $K\subset \inner\{v_{n_{4}0}=0\}$ it holds that
\begin{align}
 \sup_{t\in[0,T]}\left\|\nabla c_{n_{2:4}}(t)\right\|_{L^{\frac{4}{3}}(K)}\leq \C\left(T,\dist^{-1}(\partial K,\partial \inner\{v_{n_{4}0}=0\}),\left\|\partial_{x_i} c_{n_{3}0}\right\|_{L^{\frac{4}{3}}(\inner\{v_{n_{4}0}=0\})}\right).\label{nc24}
\end{align}
Combining \eqref{ct1} and \eqref{nc24}, we conclude using a version of the Lions-Aubin Lemma \cite[Corollary 4]{Simon}   that
\begin{align}
&c_{n_{2:4}}\underset{n_{2}\rightarrow\infty}{\rightarrow}c_{n_{3:4}}\text{ in }L^{\frac{4}{3}}(0,T;L^{\frac{4}{3}}_{loc}(\inner\{v_{n_{4}0}=0\})).\nonumber
\end{align}
We may therefore assume that 
\begin{align}
&c_{n_{2:4}}\underset{n_{2}\rightarrow\infty}{\rightarrow}c_{n_{3:4}}\text{ a.e. in }(0,T)\times \inner\{v_{n_{4}0}=0\}.\label{convc24c}
\end{align}
Combining \eqref{c2conv} and \eqref{convc24c}, we conclude with the Vitali convergence theorem that
\begin{align}
&c_{n_{2:4}}^2\underset{n_{2}\rightarrow\infty}{\rightarrow}c_{n_{3:4}}^2\text{ in }L^{1}(0,T;L^{1}(\inner\{v_{n_{4}0}=0\})).\nonumber
\end{align}
Consequently, we may pass to the limit in the distributional sense as $n_{2}\rightarrow\infty$ in \eqref{nlim10}  and obtain that $c_{n_{3:4}}$ solves 
\begin{alignat}{3}
& \partial_t c_{n_{3:4}}=\mu_cc_{n_{3:4}}-\mu_c c_{n_{3:4}}^2&&\text{ in }(0,T)\times \inner\{v_{n_{4}0}=0\}\label{limc34},\\
&c_{n_{3:4}}(0)=c_{n_{3}0}&&\text{ in }\inner\{v_{n_{4}0}=0\}.\label{limc34ini}
\end{alignat}
Since for an ODE with smooth coefficients the dependence of solutions upon initial data is continuous, we obtain with \eqref{ck0} and \eqref{limc34}-\eqref{limc34ini} that 
\begin{align}
 c_{n_{3:4}}\underset{n_{3}\rightarrow\infty}{\rightarrow}\bar{c}\text{ a.e. in }(0,T)\times \inner\{v_{n_{4}0}=0\},\label{conv34_}
\end{align}
where $\bar{c}$ solves
\begin{subequations}
\begin{alignat}{3}
& \partial_t \bar{c}=\mu_c\bar{c}-\mu_c \bar{c}^2&&\text{ in }(0,T)\times \Omega,\nonumber\\
&\bar{c}(0)=c_0&&\text{ a.e. in }\Omega.\nonumber
\end{alignat}
\end{subequations}
Combining \eqref{cconv2} and \eqref{conv34_} with  Lions lemma \cite[Lemma 1.3]{Lions}, we conclude that
\begin{align}
 c_{n_4}=\bar{c}\text{ a.e. in }(0,T)\times \inner\{v_{n_{4}0}=0\}.\label{conv34}
\end{align}
Together with \eqref{star}, \eqref{conv34} yields that
\begin{align}
 |\{c_{n_{4}}\neq\bar{c}\}\cap((0,T)\times\{v_0=0\})|\leq \epsilon_{4,n_{4}}T\underset{n_{4}\rightarrow\infty}{\rightarrow}0,\nonumber
\end{align}
so that
\begin{align}
&c_{n_{4}}\underset{n_{4}\rightarrow\infty}{\rightarrow}\bar{c}\text{ in $(0,T)\times\{v_0=0\}$ in the measure},\label{comas1}\\
&c_{n_{4}}^2\underset{n_{4}\rightarrow\infty}{\rightarrow}\bar{c}^2\text{ in $(0,T)\times\{v_0=0\}$ in the measure}.\label{comas2}
\end{align}
Combining \eqref{cconv2} and \eqref{comas1}, we conclude with the Lions lemma \cite[Lemma 1.3]{Lions} that 
 \begin{align}
c=\bar{c} \text{ a.e. in }(0,T)\times \{v_0=0\}.\nonumber
\end{align}
 Similarly, we obtain with 
 \eqref{c2conv} and \eqref{comas2} using the Vitali convergence theorem that  
 \begin{align}
\tilde{c}^2=\bar{c}^2 \text{ a.e. in }(0,T)\times \{v_0=0\}.\nonumber
\end{align} Thus, we finally arrive at
\begin{align}
c=\tilde{c}=\bar{c} \text{ a.e. in }(0,T)\times \{v_0=0\}.\nonumber
\end{align}
This concludes the global existence proof.
\section{Numerical simulations}\label{numerics}
In this section we perform numerical simulations of the system \eqref{hapto} for $N=2$ and $\Omega = (0,1)^2$. All simulations are performed via MATLAB and 
the cell-centered unstructured triangular mesh generation is implemented via the DistMesh MATLAB function package \cite{PPOSG04}. In order to obtain the numerical 
solution, we employ for the space discretization the Finite Volume Method (see e.g., \cite{eymard2000finite, BTOM04}). Due to the high nonlinearity of the system \eqref{hapto}, 
the time discretization is implemented via an explicit one-step Euler method. 
 
\subsection{Implementation}
In order to advance the piecewise constant solution $c^{(k)}\rightarrow c^{(k+1)}$ from the time level $k\in\mathbb{N}_0$ to $k+1$ we employ operator 
splitting and advance the solution with haptotactic and diffusion-reaction terms separately. Thus, the operator splitting consists of two steps:

\noindent
\textbf{Step 1}: $c^{(k)}\rightarrow c^{*}$ solving the advection problem $\partial_t c = -\nabla\cdot\left(\frac{\kappa_vc}{(1+v)^2}\nabla v\right)$ for one 
time step $\triangle t$, using $c^{(k)}$ as the initial value. We use a monotone E-flux scheme, such as the Godunov method (see e.g., \cite{BTOM04}), which is given 
by
	\begin{align*}
	c^{*}_i = c^{(k)}_i - \frac{\triangle t}{\vert\Omega_i\vert}\left( \sum\limits_{j\in A(i)}\vert\partial\Omega_{ij}\vert E_{\overrightarrow{n_{ij}}}\left(c^{(k)}_i,c^{(k)}_j\right)\right),
	\end{align*}    
	where
	\begin{itemize}
	\item $c^{(k)}_i = \frac{1}{\vert\Omega_i\vert}\int_{\Omega_i}c^{(k)}$ is the average value of the piecewise constant 
	solution $c^{(k)}$ over the triangle $\Omega_i$ (with tessellation $\bigcup_{i\in I}\Omega_i=\Omega$, $I$ being an index set) at the time level $k$,
	\item $A(i)$ is an index set of the neighboring triangles of $\Omega_i$,
	\item $\partial\Omega_{ij}$ is the boundary edge between triangles $\Omega_i$ and $\Omega_j$,
	\item $E_{\overrightarrow{n_{ij}}}\left(c^{(k)}_i,c^{(k)}_j\right)$ is the Godunov flux from $\Omega_i$ to $\Omega_j$, $\overrightarrow{n_{ij}}$ is the outward unit normal, 
	pointing out of $\Omega_i$ and into $\Omega_j$. 
	\end{itemize}
	The Godunov flux is given by:\\
	\[
	    E_{\overrightarrow{n_{ij}}}\left(c^{(k)}_i,c^{(k)}_j\right)= 
	\begin{cases}
	    \min\limits_{u\in \left[c^{(k)}_i,c^{(k)}_j\right]}f(u)n_x+g(u)n_y,& \text{if } c^{(k)}_i\leq c^{(k)}_j\\
	    \max\limits_{u\in \left[c^{(k)}_j,c^{(k)}_i\right]}f(u)n_x+g(u)n_y,& \text{otherwise.}
	\end{cases}
	\]
	Thereby, $n_x$ and $n_y$ denote the $x$ and $y$ components of the unit normal $\overrightarrow{n_{ij}}$, respectively. The functions $f$ and $g$ are given by:
	\begin{align*}
	f(u) &= u\cdot \frac{\kappa_v}{(1+v^{(k)})^2}\partial_xv^{(k)}\rvert_{\partial\Omega_{ij}}\\
	g(u) &= u\cdot \frac{\kappa_v}{(1+v^{(k)})^2}\partial_yv^{(k)}\rvert_{\partial\Omega_{ij}},
	\end{align*}  
	where
	\begin{align*}
	\partial_xv^{(k)}\rvert_{\partial\Omega_{ij}}& = \frac{v^{(k)}_j-v^{(k)}_i}{\vert x_j-x_i\vert},\\
	\partial_yv^{(k)}\rvert_{\partial\Omega_{ij}} &= \frac{v^{(k)}_j-v^{(k)}_i}{\vert y_j-y_i\vert},\\
	{\frac{\kappa_v}{(1+v^{(k)})^2}}\rvert_{\partial\Omega_{ij}} &= \frac{1}{2}\left(\frac{\kappa_v}{(1+v_i)^2}+\frac{\kappa_v}{(1+v_j)^2}  \right),
	\end{align*}   
	with $(x_i,y_i)$ being cell center coordinates of triangle $\Omega_i$, $v^{(k)}_i$ is cell average at the time level $k$ defined similarly as above.

\noindent	
	\textbf{Step 2}: $c^{*}\rightarrow c^{(k+1)}$ solving the reaction-diffusion problem $\partial_t c = \nabla \cdot\left(\frac{\kappa_cvc}{1+vc}\nabla c\right) 
	+ \mu_cc(1-c-\eta v)$ for one time step $\triangle t$, thereby using $c^{(*)}$ as the initial value. The scheme is given by
		\begin{align*}
		c^{k+1}_i = c^{*}_i + \frac{\triangle t}{\vert\Omega_i\vert}
		\left( \sum\limits_{j\in A(i)}\vert\partial\Omega_{ij}\vert D_{\overrightarrow{n_{ij}}}\left(c^{(*)}_i,c^{(*)}_j\right)\right) + \triangle t P^{(*)}_i,
		\end{align*}
	 	where
	 	\begin{align*}
	 	D_{\overrightarrow{n_{ij}}}\left(c^{(*)}_i,c^{(*)}_j\right) &= \frac{\kappa_cv^{(k)}c^{(*)}}{1+v^{(k)}c^{(*)}}
	 	\left( \partial_xc^{(*)}n_x+\partial_yc^{(*)}n_y\right)\rvert_{\partial\Omega_{ij}},\\
	 	P^{(*)}_i &= \mu_cc^{(*)}_i\left(1-c^{(*)}_i-\eta v^{(k)}_i\right).
	 	\end{align*}
		Here the function evaluations at the boundary edge $\partial\Omega_{ij}$ are approximated similarly as above. 

\noindent		
The solution $v^{(k+1)}_i$ is obtained by using one-step time marching: 		
\begin{align*}
v^{(k+1)}_i &= v^{(k)}_i +\triangle t\mu_vv^{(k)}_i\left(1-v^{(k)}_i\right)-\triangle t\lambda v^{(k)}_ic^{(k)}_i.
\end{align*}

\noindent
We simulate the initial ECM density by uniformly distributed random numbers on the interval $(0,1)$, i.e., we have:
\begin{align*}
v_0(x,y) \sim \mathcal{U}(0,1),\phantom{+} (x,y)\in\Omega.
\end{align*} 
The initial tumor cell density is given by the following:
\begin{align*}
c_0(x,y) = exp\left(-\frac{(x-0.5)^2+(y-0.5)^2)}{2\epsilon^2}\right), \phantom{+} (x,y)\in\Omega,
\end{align*}  
where we took $\epsilon=0.08$. That is, $c_0$ is a bell-shaped curve centered at $(0.5,0.5)$. The plots of $c_0$ and $v_0$ are given in {\it Figure \ref{fig:initial}}.

\begin{figure}[h]

\begin{subfigure}{0.5\textwidth}
\centering
\includegraphics[width=0.7\linewidth, height=5cm]{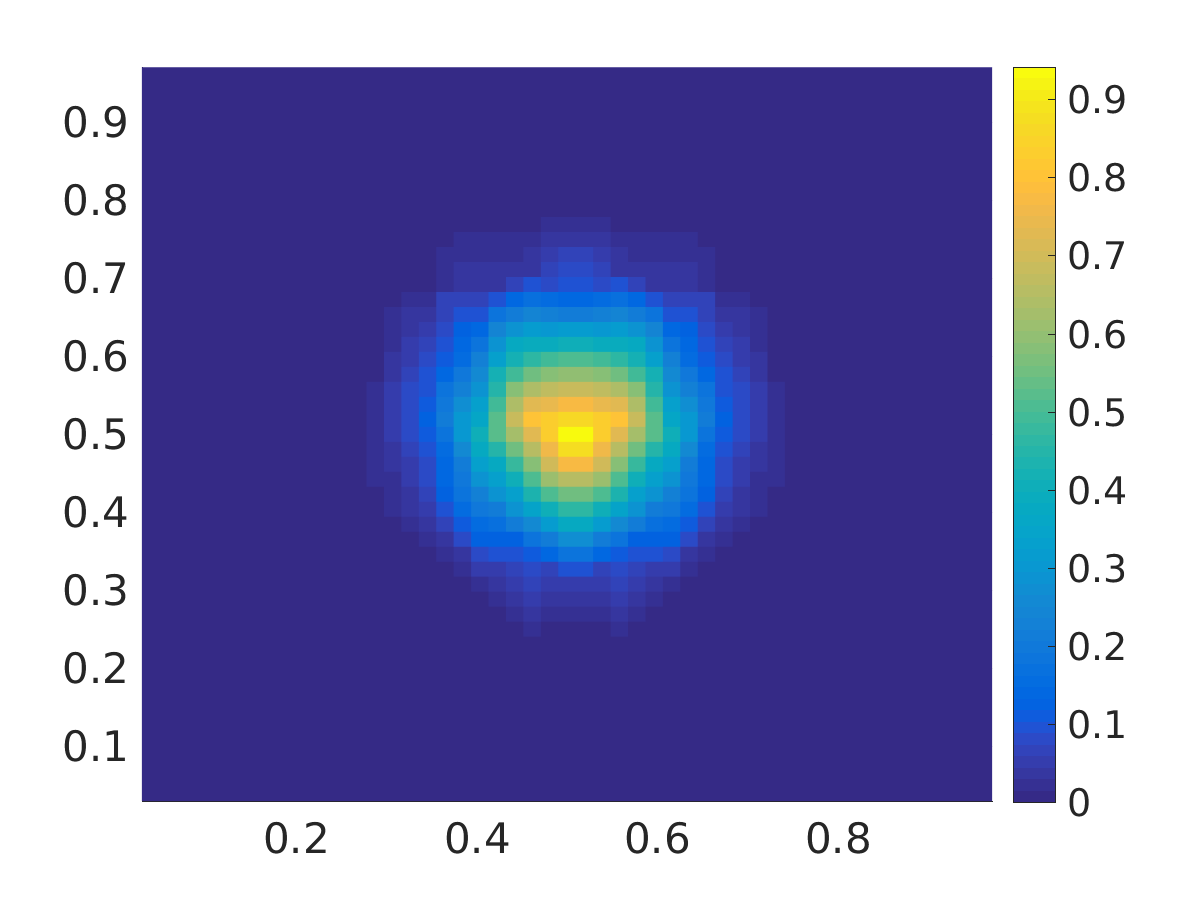} 
\caption{Initial cancer cell density $c_0$}
\end{subfigure}
\begin{subfigure}{0.5\textwidth}
\centering
\includegraphics[width=0.7\linewidth, height=5cm]{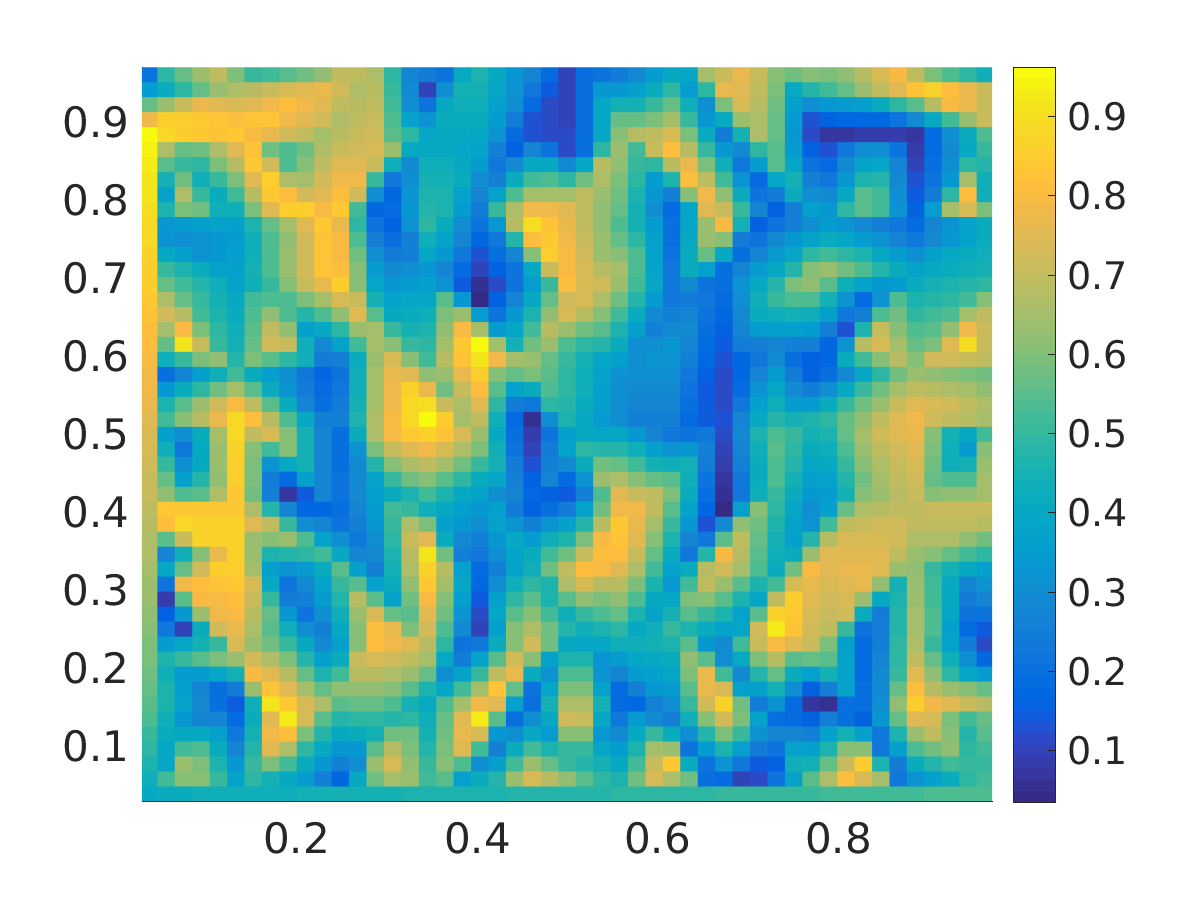}
\caption{Initial density $v_0$ of tissue fibers}
\end{subfigure}
 
\caption{Initial conditions}
\label{fig:initial}
\end{figure} 

\noindent
The values of the model parameters used for solving the system \eqref{hapto} are given below:
\begin{align*}
\kappa_c &= 10^{-3}, &\kappa_v &= 1,	&\mu_c &= 0.5, &\mu_v&=0.02, &\lambda &=0.1. 
\end{align*}
\noindent
These parameter values are in agreement with those estimated in \cite{JMGL06}. Since cancer cells grow much faster than healthy tissue can be restructured, 
$\mu_v$ was taken to be a 
fraction of the cancer cell proliferation rate $\mu_c$.

\noindent
We also perform numerical simulations for a version of the Equation \eqref{c} with nondegenerate diffusion:
\begin{align*}
&\partial_t \tilde{c}=\nabla\cdot\left(\frac{\kappa_c }{1+\tilde{v}\tilde{c}}\nabla \tilde{c}-\frac{\kappa_v \tilde{c}}{(1+\tilde{v})^2}\nabla \tilde{v}\right)+\mu_c\tilde{c}(1-\tilde{c}-\eta \tilde{v})&&\text{ in }\R^+\times\Omega,\\
&\partial_t \tilde{v}=\mu_v\tilde{v}(1-\tilde{v})-\lambda \tilde{v}\tilde{c}&&\text{ in }\R^+\times\Omega,\\
&\frac{\kappa_c }{1+\tilde{v}\tilde{c}}\partial_{\nu} \tilde{c}-\frac{\kappa_v \tilde{c}}{(1+\tilde{v})^2}\partial_{\nu} \tilde{v}=0&&\text{ in }\R^+\times\partial\Omega,\\
&\tilde{c}(0)=\tilde{c}_0=c_0,\ \tilde{v}(0)=\tilde{v}_0=v_0 &&\text{ in }\Omega.
\end{align*}
\subsection{Results}

\noindent
The simulation results are shown in {\it Figures \ref{cancer_f}} and {\it \ref{tissue_f}}. We observe that in the nondegenerate case the cancer cells are able to diffuse quite fast throughout 
the whole domain. In particular, 
the tumor cells can surpass regions of low or even (locally) vanishing ECM density and can invade their surroundings, thereby degrading the tissue in a more effective way. 
Due to this extended ECM deterioration flattening the fiber density profile and to the higher diffusivity throughout the domain (not restricted by gaps in the tissue), 
the haptotactic component of migration is outweighed by random motility. Therefore, the behavior of cancer cells with nondegenerate diffusion is much more aggressive than 
in the case with degenerate diffusion, where on the one hand the tumor cells are locally trapped between the regions with gaps ($v=0$) and on the other hand no diffusion 
takes place in the regions with $c=0$ until the tumor growth (via proliferation) did not repopulate the regions where cancer cells were lacking. Hence, cancer invasion 
models with nondegenerate diffusion might overestimate the extent of a tumor, especially if the latter is situated in a rather sparse tissue.
\begin{figure}[!tbph]\label{figures1}
  \centering
  \begin{minipage}[b]{0.45\textwidth}
  	\centering
  	\begin{subfigure}{1\textwidth}
  		\includegraphics[width=0.48\linewidth, height=4cm]{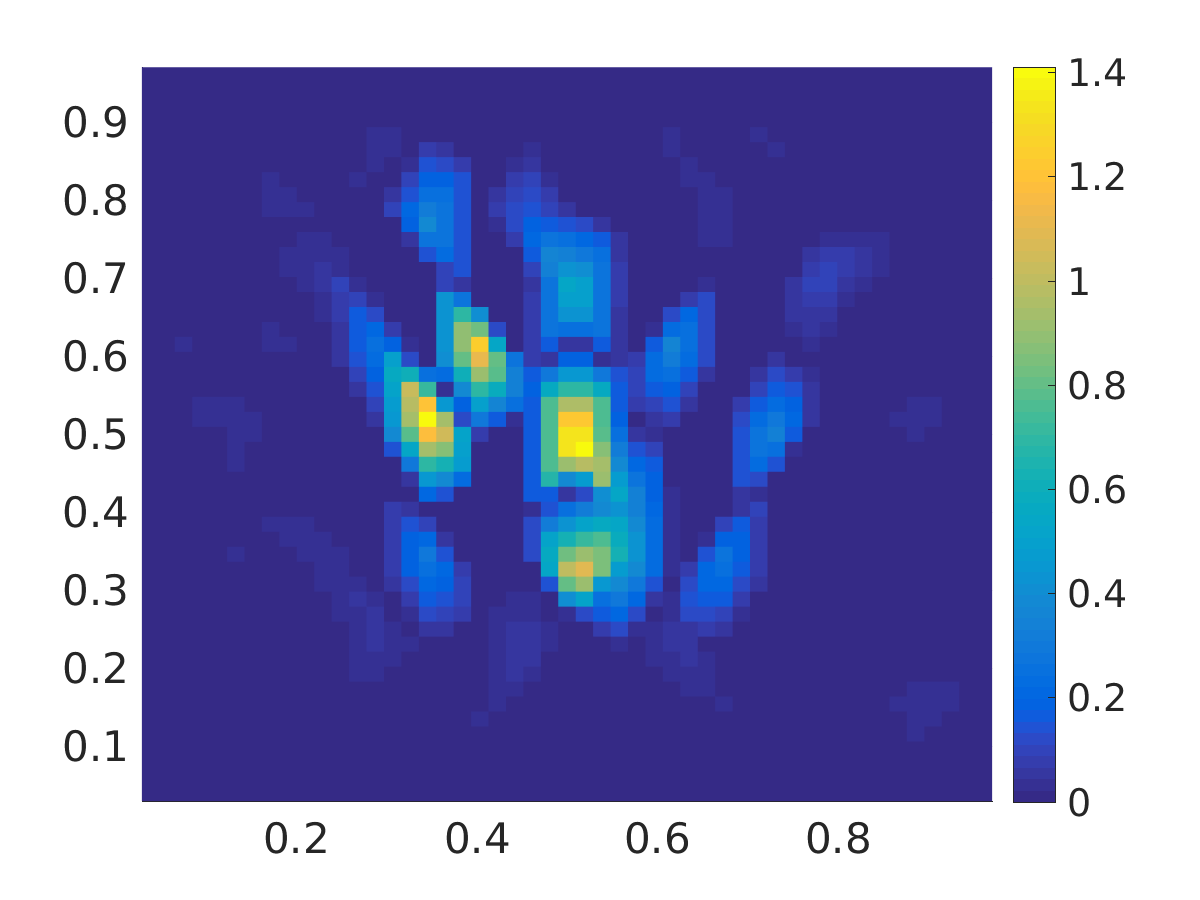} 
  		\includegraphics[width=0.48\linewidth, height=4cm]{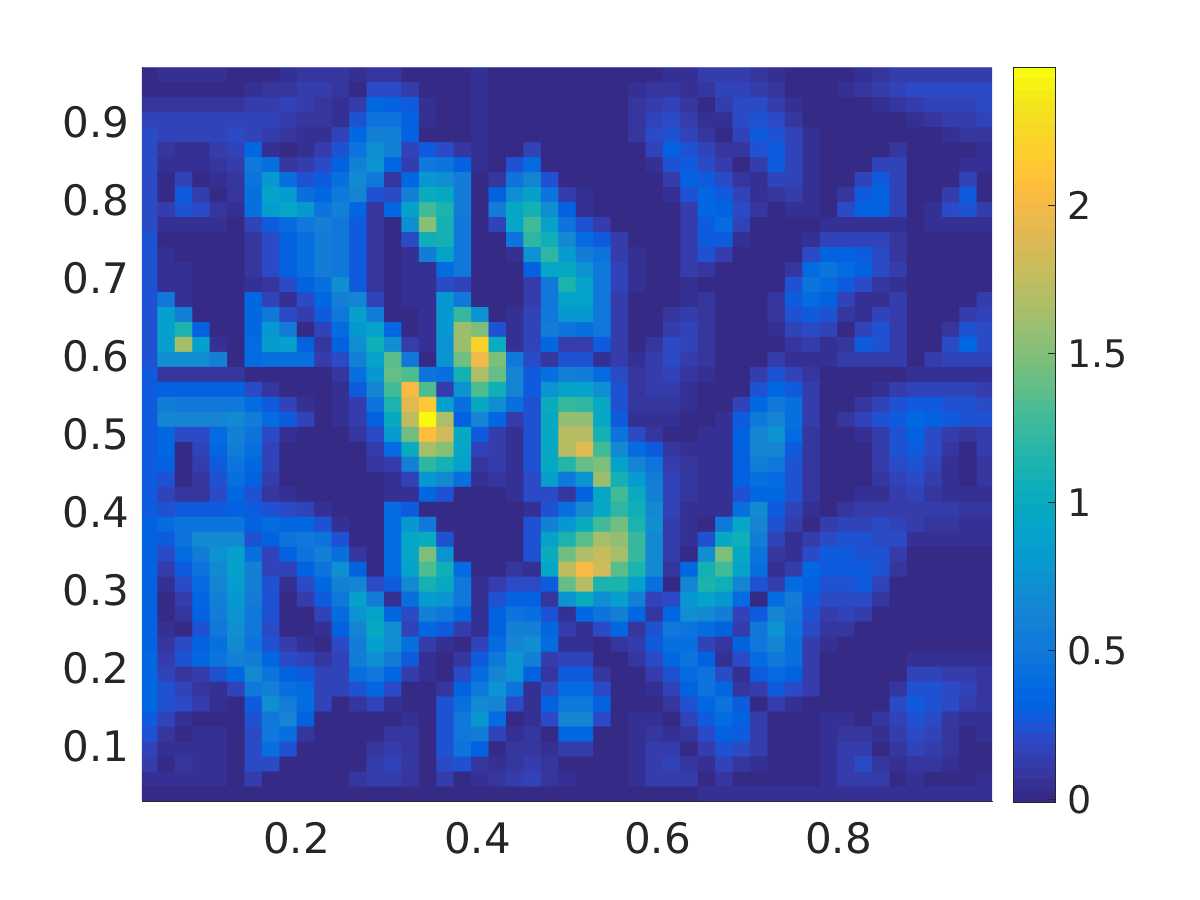} 
  	\end{subfigure}
  	\begin{subfigure}{1\textwidth}
  	  	\includegraphics[width=0.48\linewidth, height=4cm]{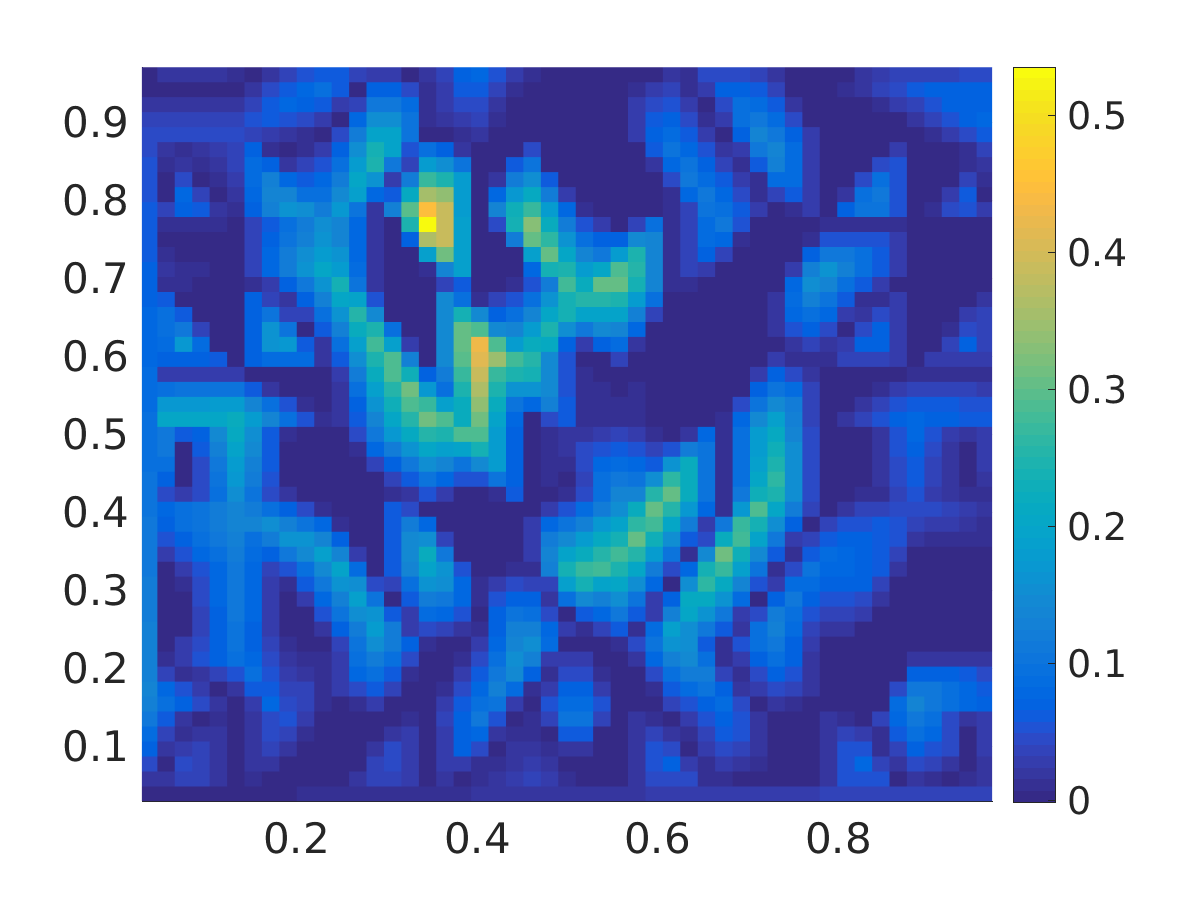} 
  	  	\includegraphics[width=0.48\linewidth, height=4cm]{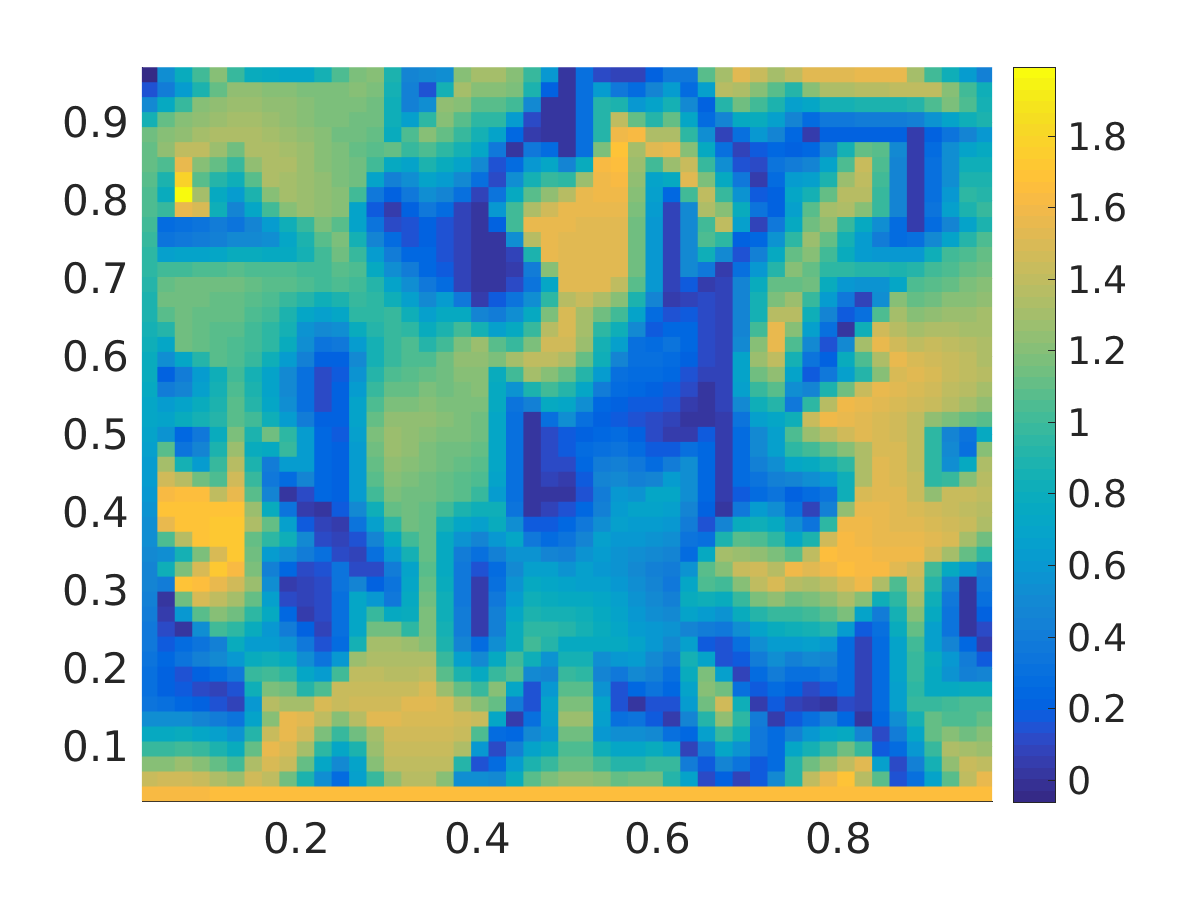} 
  	\end{subfigure}
  	\begin{subfigure}{1\textwidth}
  	  	\includegraphics[width=0.48\linewidth, height=4cm]{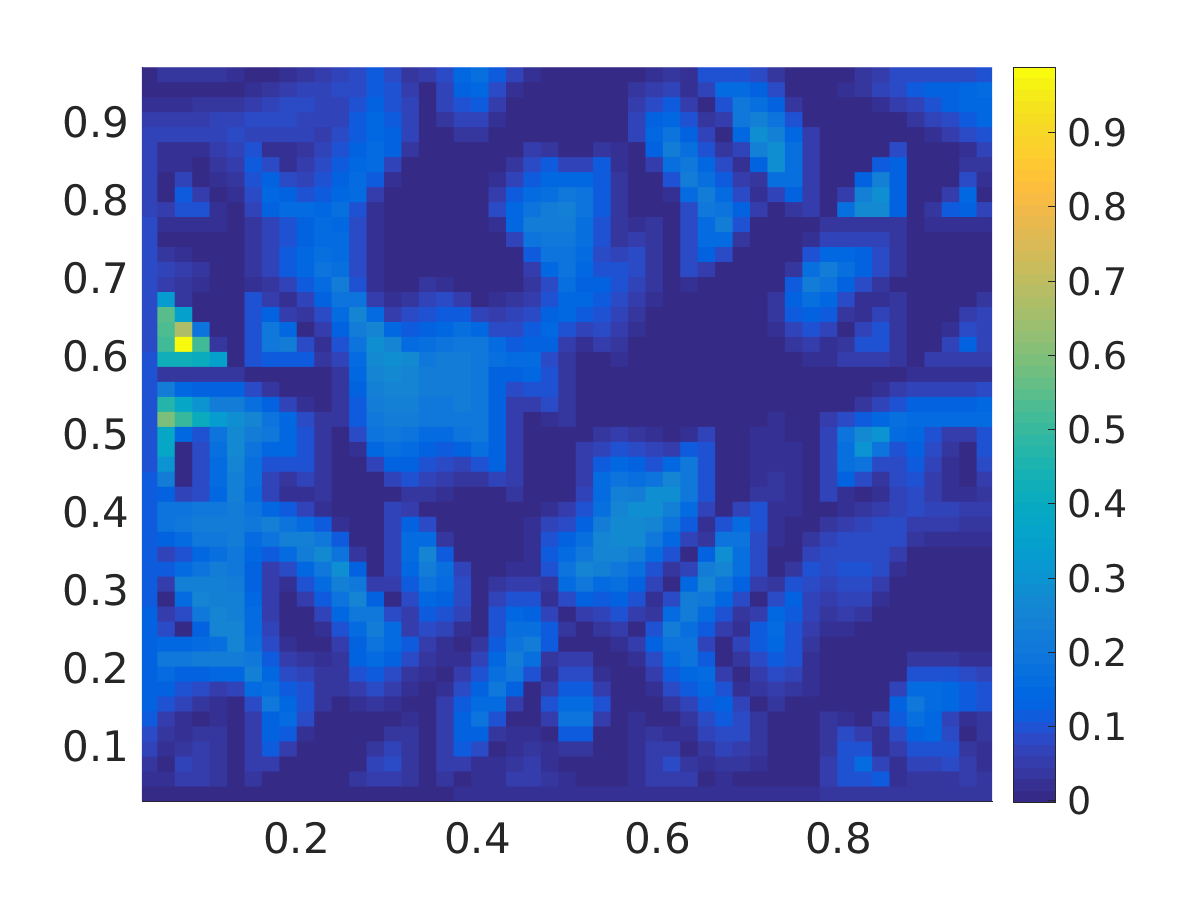} 
  	  	\includegraphics[width=0.48\linewidth, height=4cm]{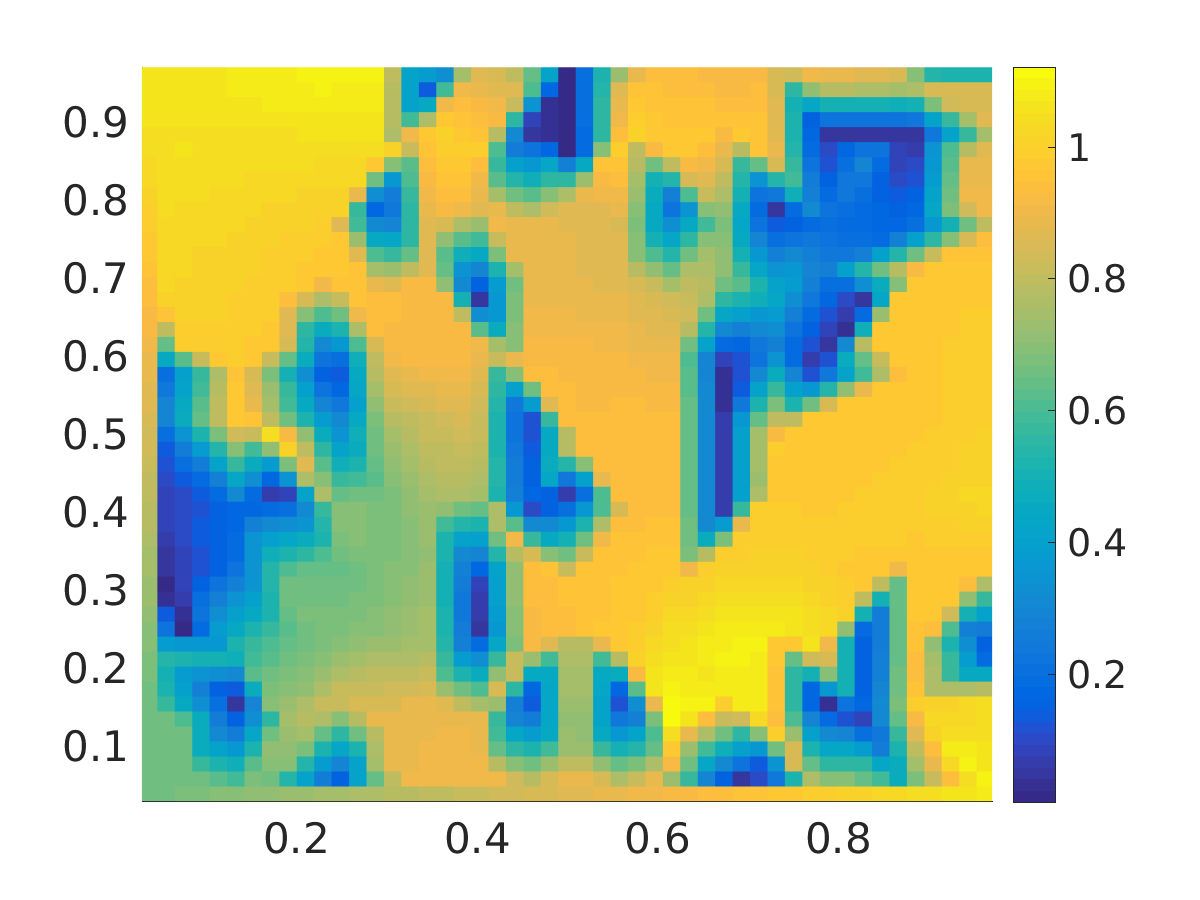} 
  	\end{subfigure}
	\begin{subfigure}{1\textwidth}
   	  	 \includegraphics[width=0.48\linewidth, height=4cm]{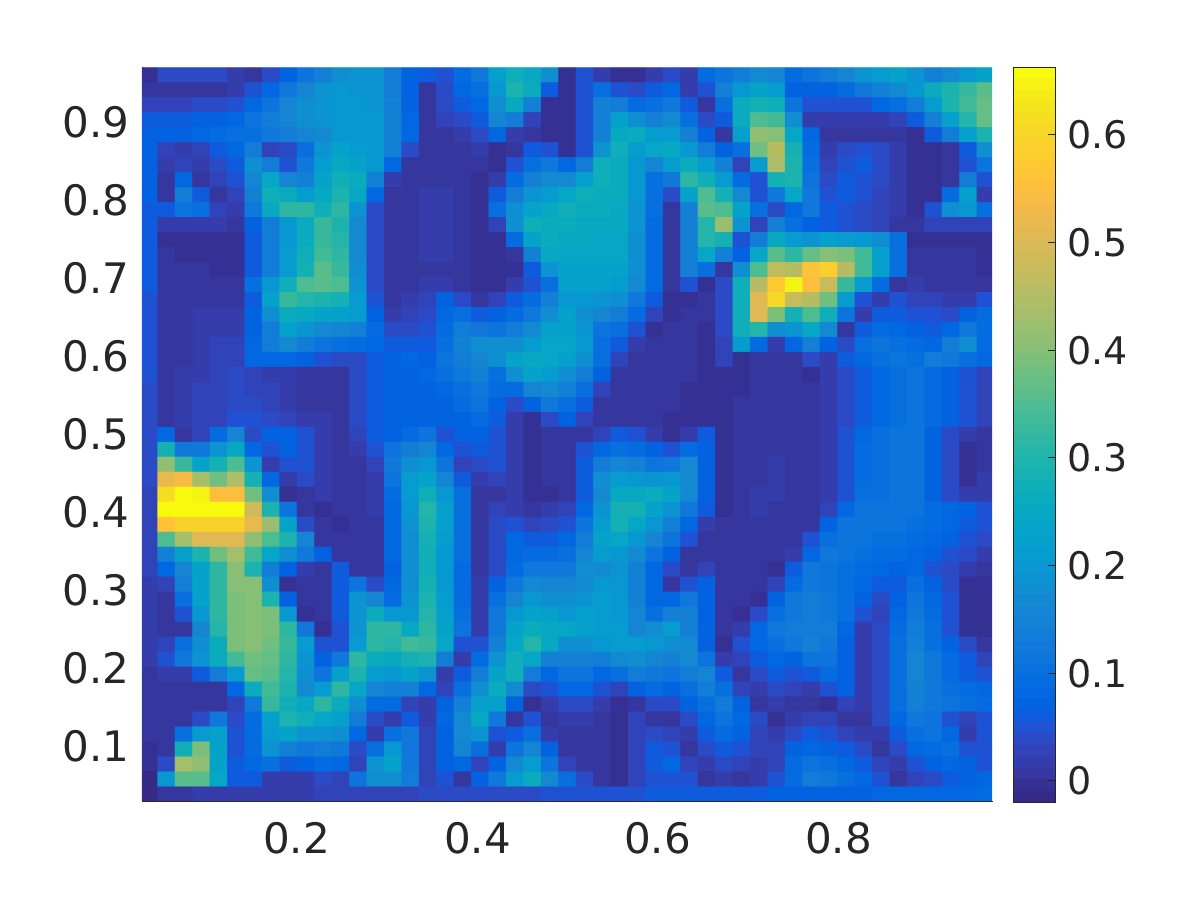} 
   	  	 \includegraphics[width=0.48\linewidth, height=4cm]{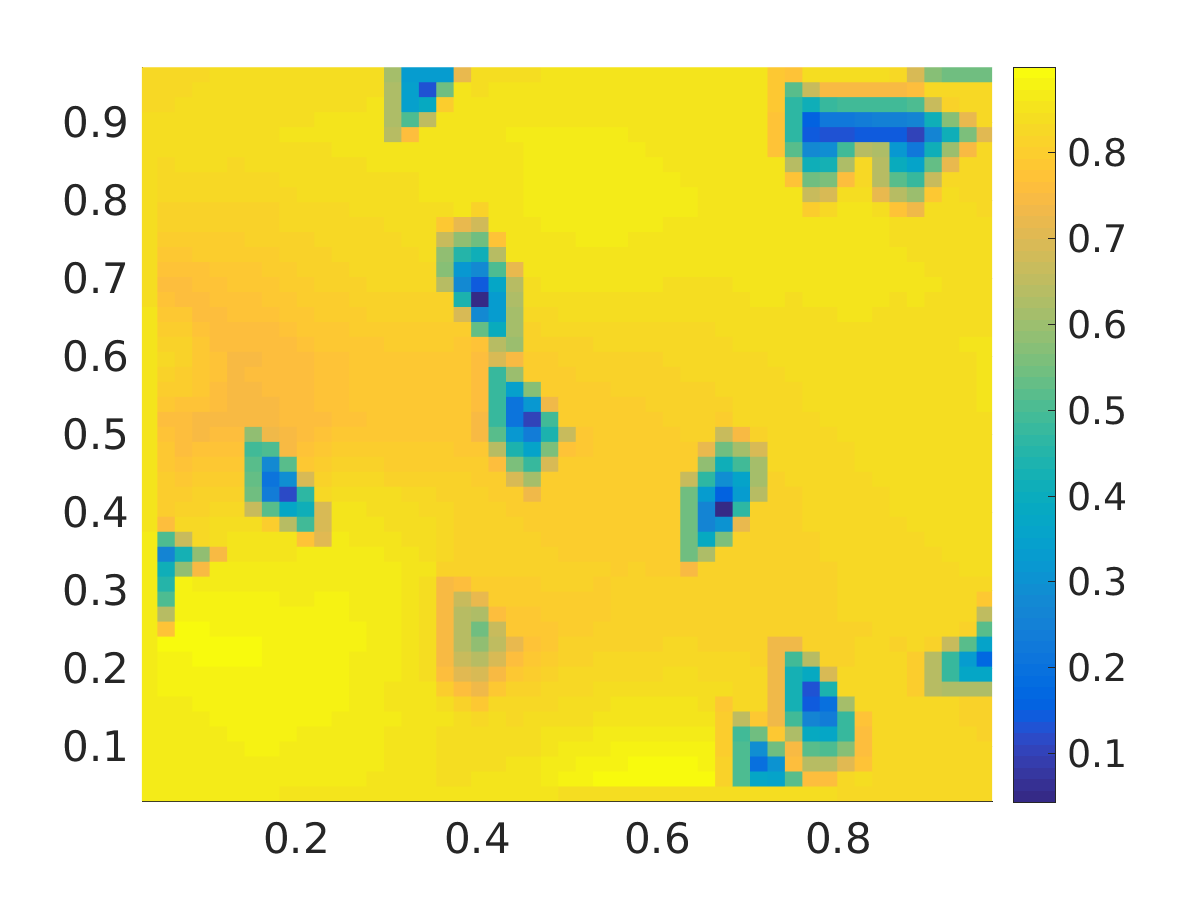} 
   	\end{subfigure}
  	\caption{Simulation results for the tumor cell density with degenerate diffusion (left column) and with non-degenerate diffusion (right column) at 1, 6, 12, and 24 weeks}
	\label{cancer_f}
  \end{minipage}
  \hfill
  \begin{minipage}[b]{0.45\textwidth}
	\centering
    \begin{subfigure}{1\textwidth}
      	\includegraphics[width=0.48\linewidth, height=4cm]{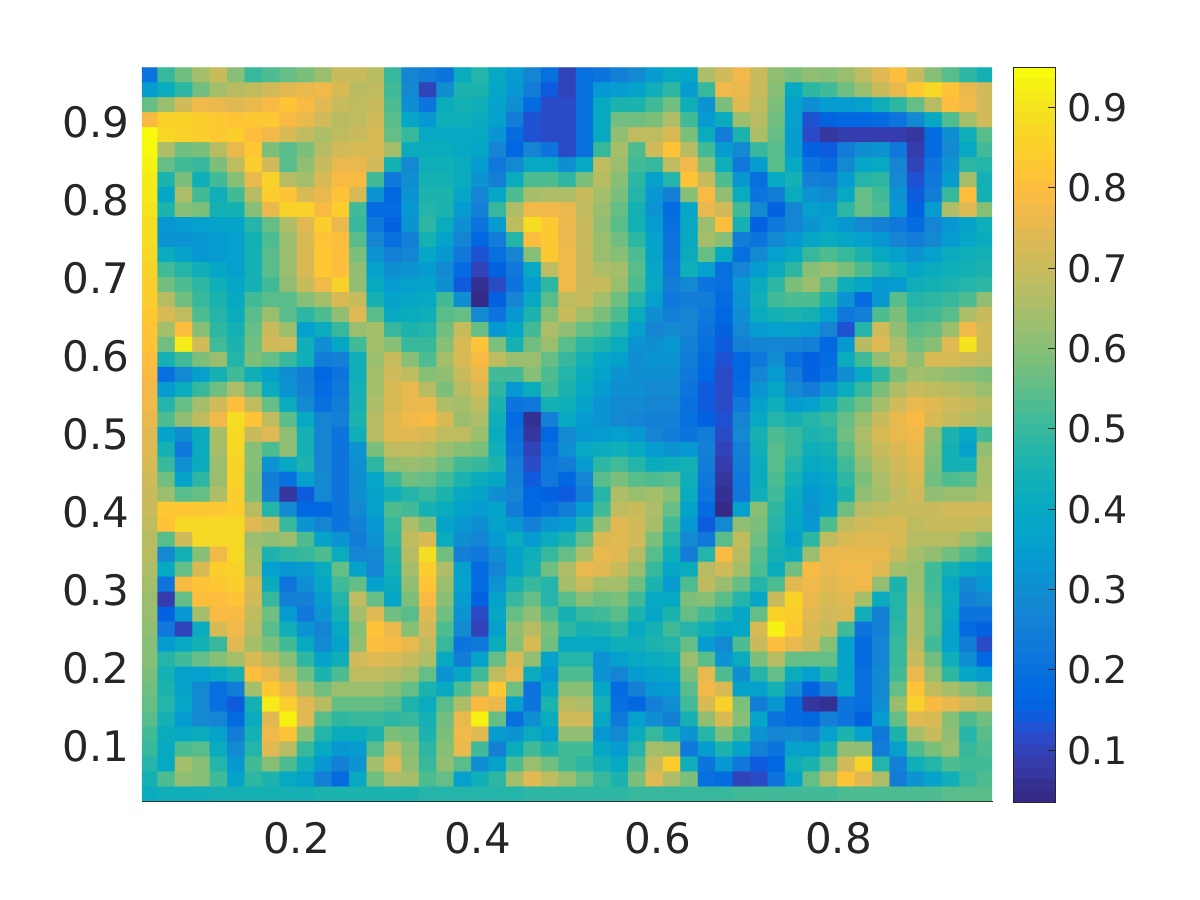} 
      	\includegraphics[width=0.48\linewidth, height=4cm]{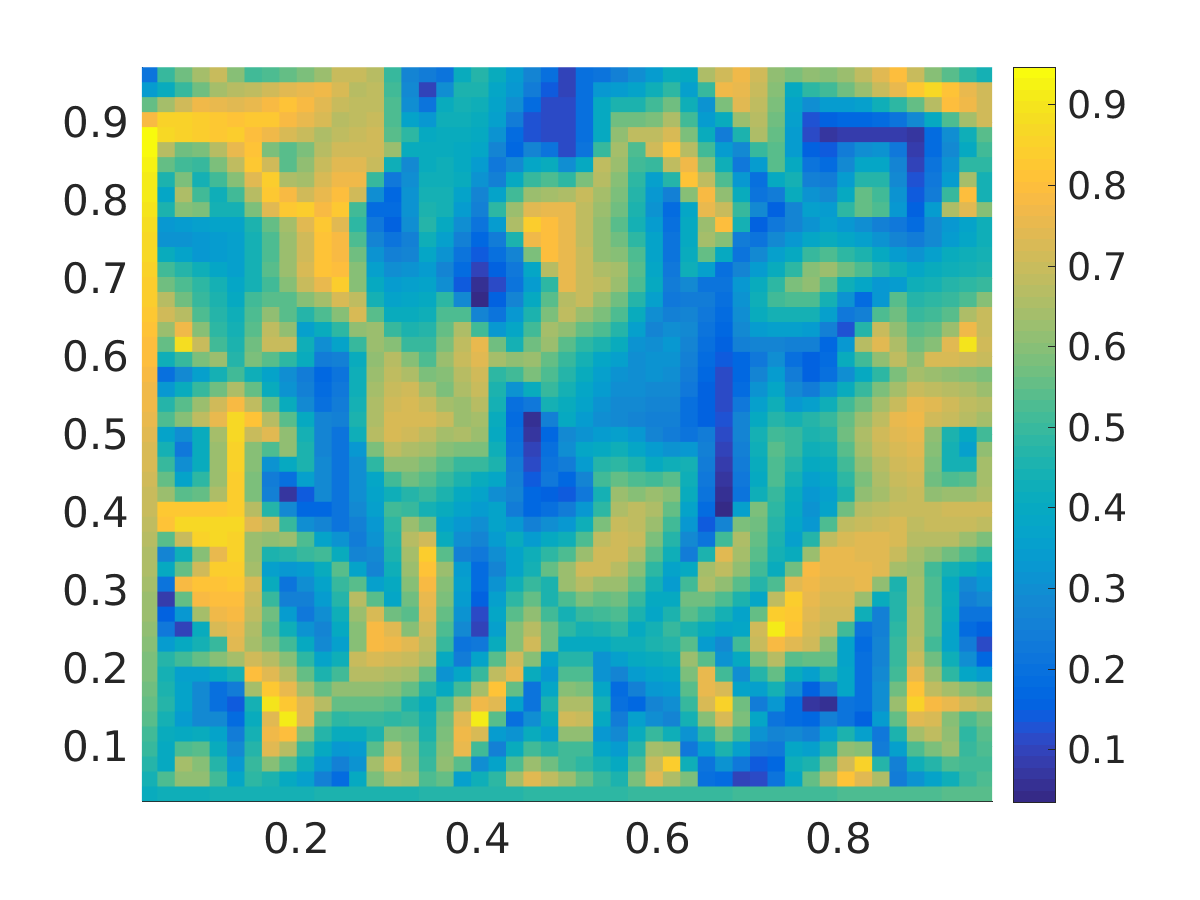} 
    \end{subfigure}
    \begin{subfigure}{1\textwidth}
         \includegraphics[width=0.48\linewidth, height=4cm]{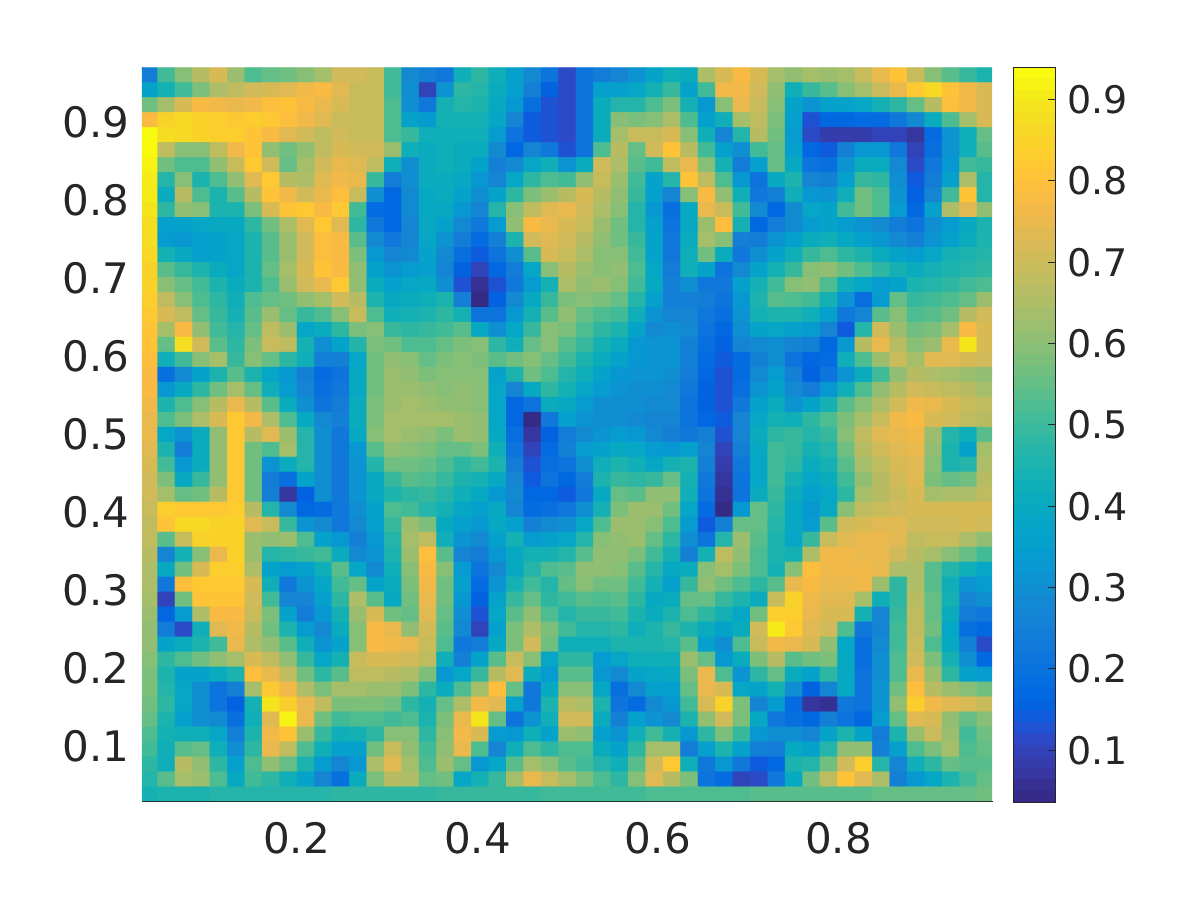} 
         \includegraphics[width=0.48\linewidth, height=4cm]{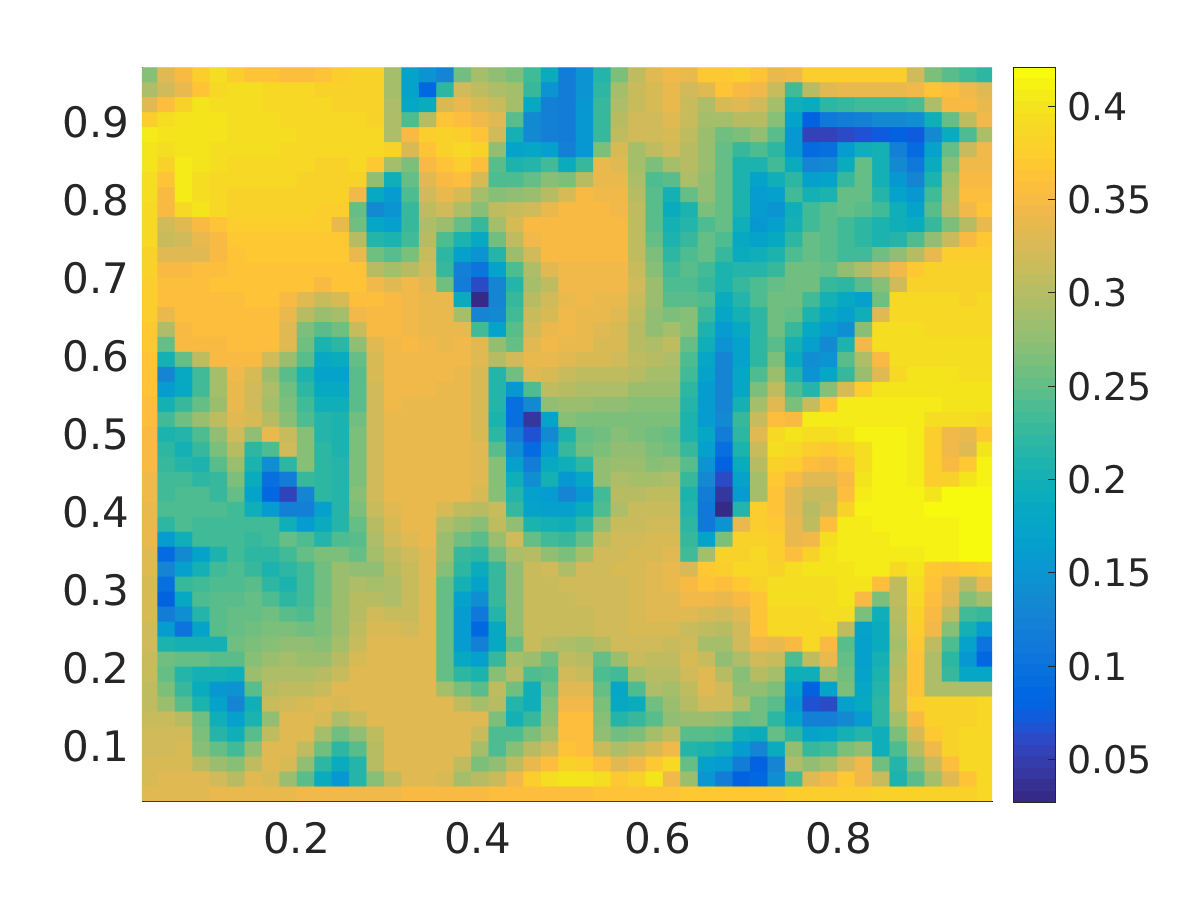} 
    \end{subfigure}
    \begin{subfigure}{1\textwidth}
         \includegraphics[width=0.48\linewidth, height=4cm]{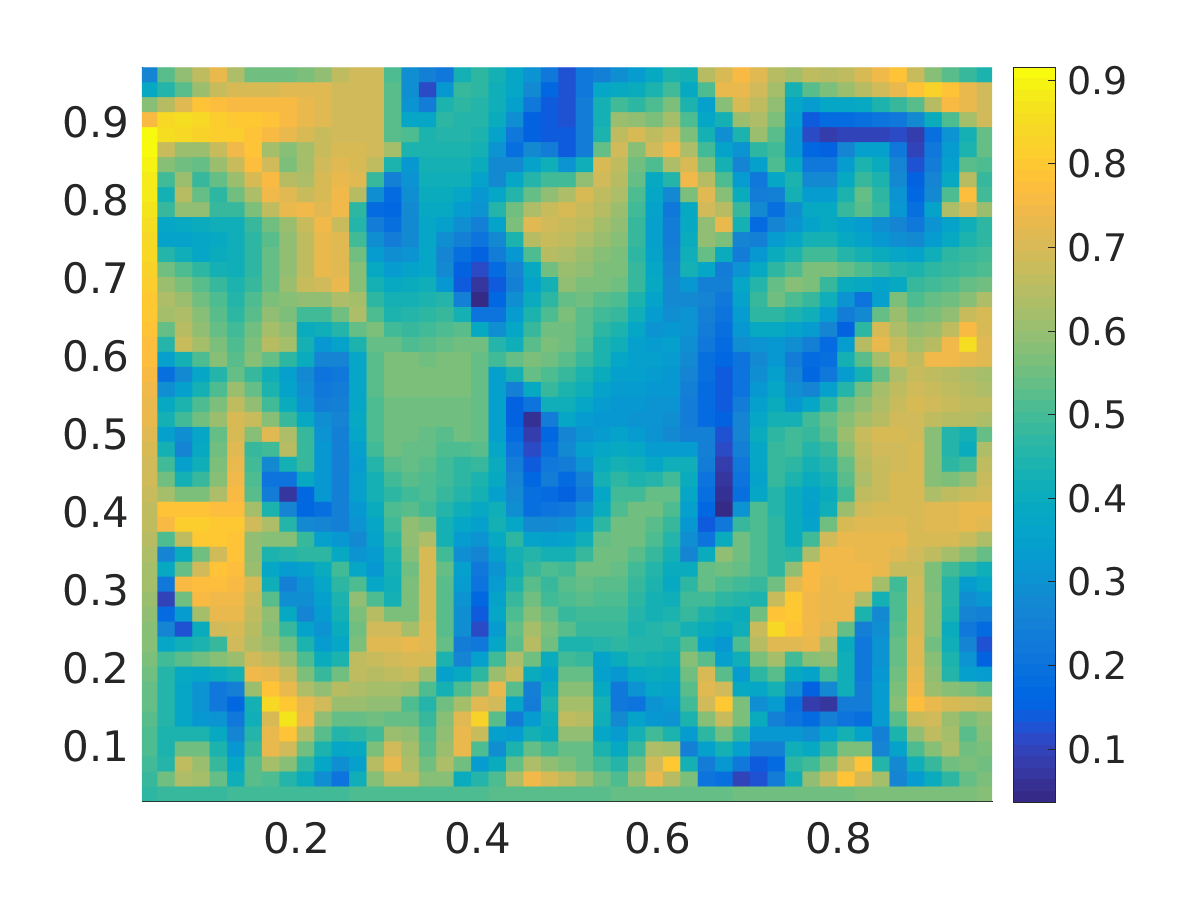} 
         \includegraphics[width=0.48\linewidth, height=4cm]{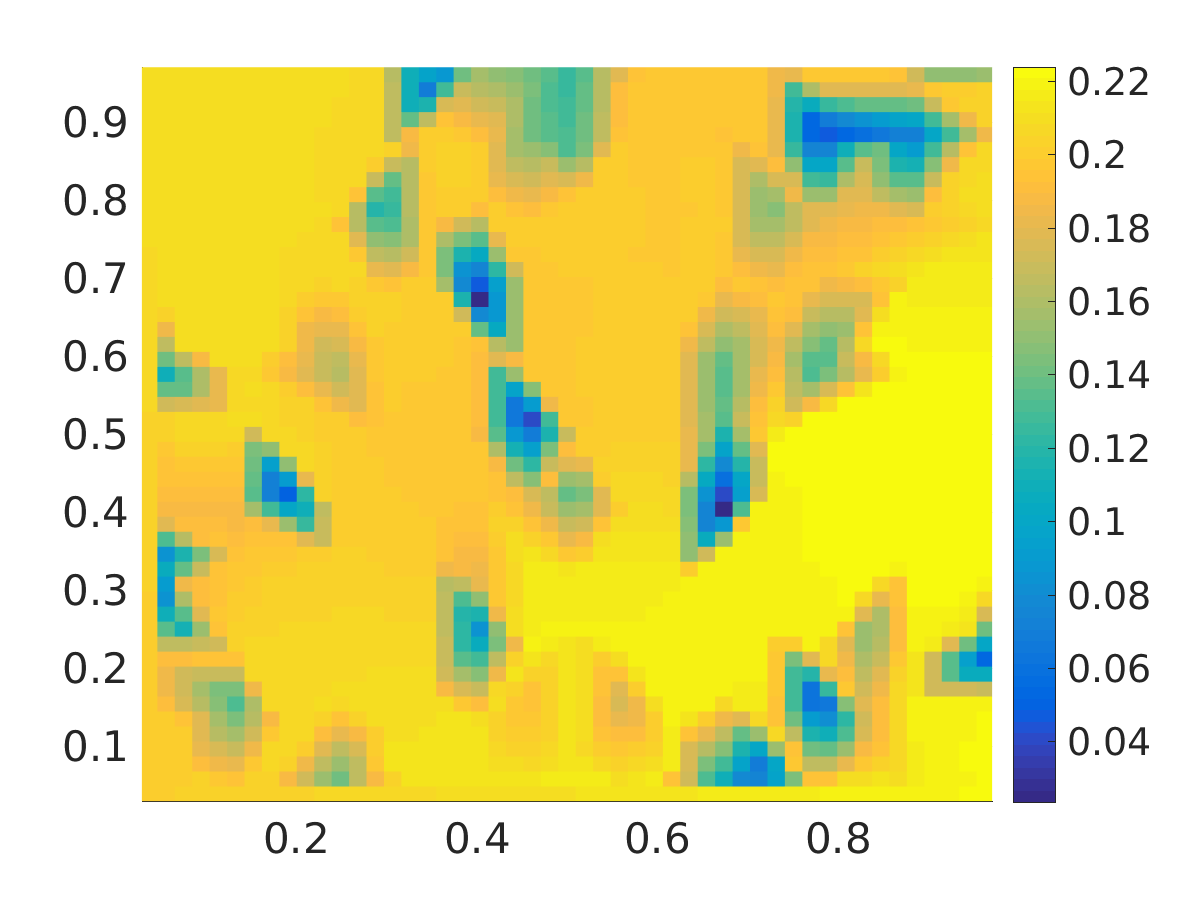} 
     \end{subfigure}
     \begin{subfigure}{1\textwidth}
           \includegraphics[width=0.48\linewidth, height=4cm]{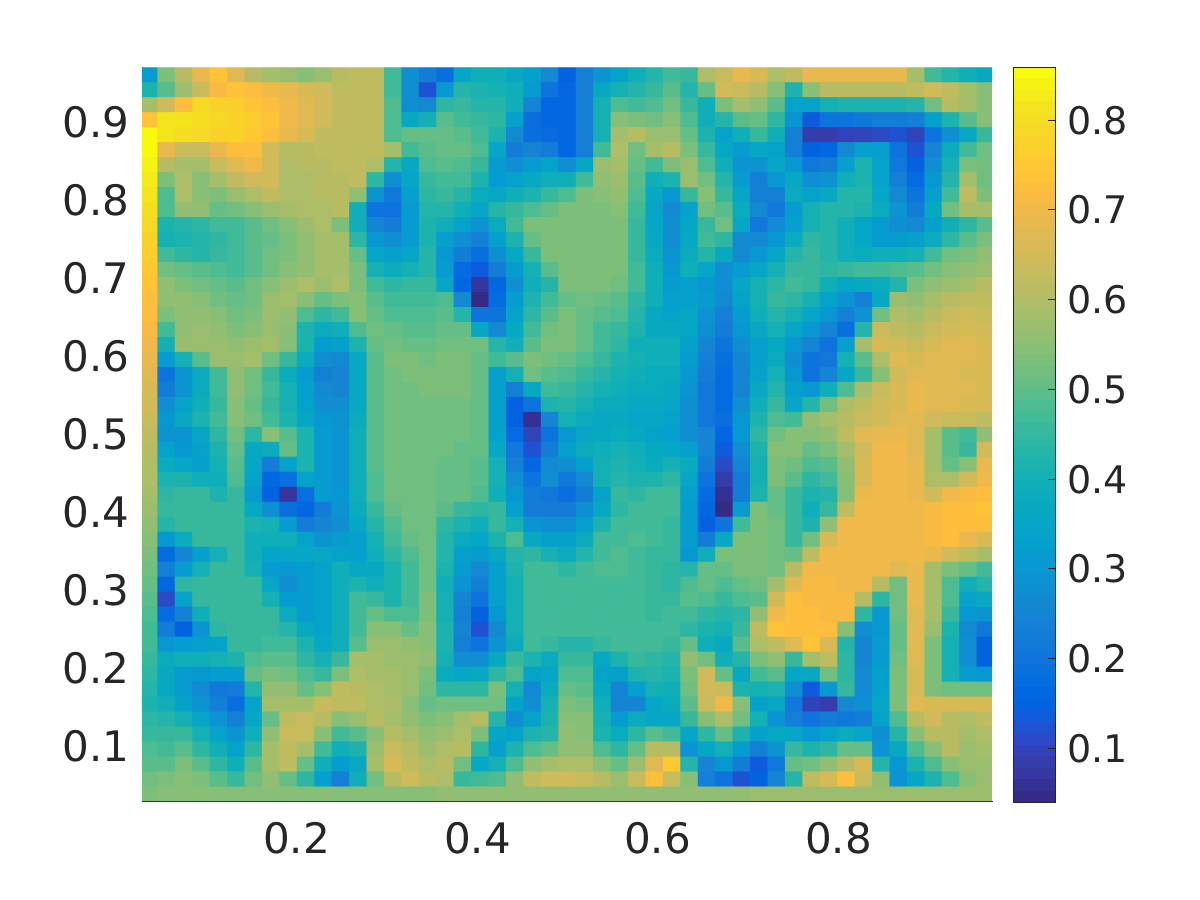} 
           \includegraphics[width=0.48\linewidth, height=4cm]{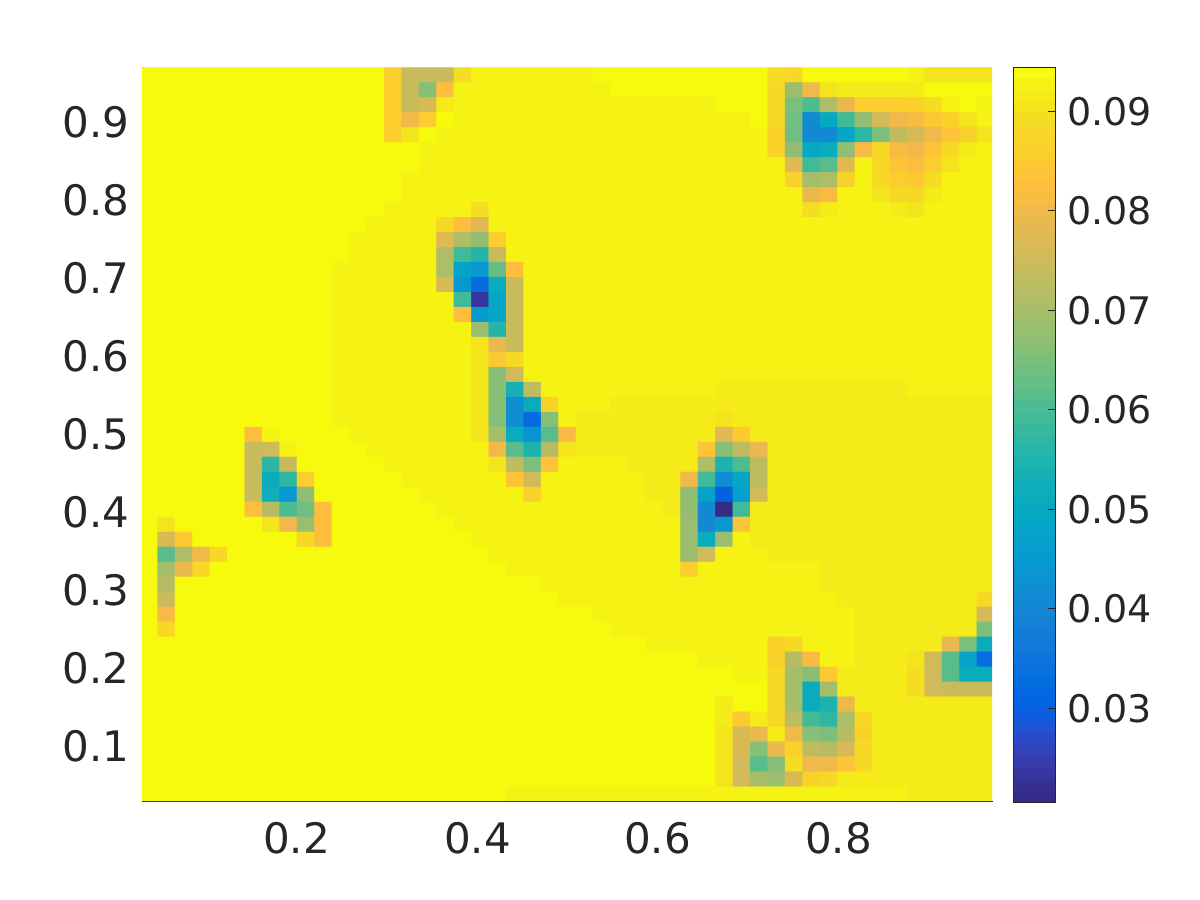} 
     \end{subfigure}
    \caption{Simulation results for the ECM density with degenerate diffusion (left column) and with non-degenerate diffusion (right column) at 1, 6, 12, and 24 weeks}
    \label{tissue_f}
  \end{minipage}
\end{figure}

\section{Discussion}\label{discuss}

We proposed a reaction-diffusion-haptotaxis model for tumor cell migration through tissue allowing for degenerate diffusion. The source of degeneracy is twofold: it can be due to 
the cancer cell density becoming zero and/or it can be triggered by the (locally) vanishing density of tissue fibers. Actually, both diffusion and haptotaxis terms can 
degenerate, but the haptotaxis coefficient can only become zero whenever the tumor cell density is vanishing. Further models with degenerate diffusion have been considered 
and analyzed e.g., in \cite{wang-winkler-wrzosek-12,eberl-efendiev-zhigun-14,horger-etal-14,zheng-etal-16}, however, to our knowledge the present type of degeneracy is new in 
the context of (hapto)taxis. Particularly the presence of a possibly vanishing function $v$ in the numerator of the diffusion coefficient brought about some serious mathematical challenges, due to 
the absence of diffusion in the equation satisfied by the density of ECM fibers.\\

\noindent
We proved the global existence of a solution to the higly 
nonlinear system coupling a PDE for the cancer cell density with an ODE for the ECM density. Precisely, {\it Theorem~\ref{maintheo}} ensures the existence of at least 
one weak solution to \eqref{hapto}. It remains open whether this solution is unique and whether there  exists a solution which is globally or locally uniformly bounded. \\

\noindent
The model in \cite{SSW} involved a nondegenerate diffusion coefficient of the form $D_c(c,v)=\frac{\kappa }{1+cv}$, with $\kappa $ decreasing 
or nearly constant in time. Hence, the diffusivity was assumed there to decrease for strong interactions between cells and tissue. Here we considered a limited increase of the form 
$D_c(c,v)=\frac{\kappa _ccv}{1+cv}$, which can lead to the mentioned twofold degeneracy of diffusion for the tumor cells. In {\it Section \ref{numerics}} we compared via simulations 
the behavior of 
haptotaxis-only models involving the two choices of diffusion coefficients \footnote{and featuring the same haptotactic coefficient $\chi (c,v)=\frac{\kappa _vc}{(1+v)^2}$, 
which differs from the one in \cite{SSW}}. It turned out that the latter choice predicts slower invasion of the tumor and local formation of cell aggregates in the 
proximity of gaps in the tissue; moreover, the degradation of ECM fibers seems to be weaker with the degenerate model than with its nondegenerate counterpart. This suggests the 
case with degenerate diffusion is describing a less aggressive tumor behavior.

\section*{Acknowledgement} A.U. is supported by the German Academic Exchange Service (DAAD).

\phantomsection
\printbibliography
\begin{appendices}
\section{}\label{}
The following lemma is a generalisation of the Lions lemma \cite[Lemma 1.3]{Lions} and the known result on weak-strong convergence for member-by-member products.
\begin{Lemma}[Weak-a.e. convergence]\label{LemA1}
 Let $\Omega$ be a measurable subset of $\R^N$ with finite measure. Let $f,f_n:\Omega\rightarrow\R$, $n\in\N$ be measurable functions and $g,g_n\in L^1(\Omega)$, $n\in\N$. Assume further that $f_n\underset{n\rightarrow\infty}{\rightarrow} f$ a.e. in $\Omega$ and $g_n\underset{n\rightarrow\infty}{\rightharpoonup}g$, $f_ng_n\underset{n\rightarrow\infty}{\rightharpoonup}\xi$ in $L^1(\Omega)$. Then, it holds that $\xi=fg$ a.e. in $\Omega$.
 \end{Lemma}

 \begin{proof}
Since $f$ is a measurable function, the sets $\Omega_k:=\{|f|\leq k\}$, $k\in\N$, are measurable and $|\Omega\backslash \underset{k\in\N}{\bigcup}\Omega_k|=0$.
  Further, due to the Egorov's theorem, there exists for each pair $k,m\in\N$ a measurable subset $\Omega_{k,m}$ of $\Omega_k$ such that $|\Omega_k\backslash\Omega_{k,m}|\leq\frac{1}{m}$ and $f_n\underset{n\rightarrow\infty}{\rightarrow} f$ uniformly in $\Omega_{k,m}$.
  Thus, we have that $||f||_{L^{\infty}(\Omega_{k,m})}\leq k$ and $||f_n-f||_{L^{\infty}(\Omega_{k,m})}\underset{n\rightarrow\infty}{\rightarrow} 0$.
  Since $g_n\underset{n\rightarrow\infty}{\rightharpoonup}g$ in $L^1(\Omega)$, the same holds in  $L^1(\Omega_{k,m})$. As a weakly converging sequence, $\{g_n\}_{n\in\N}$ is uniformly bounded: $\underset{n\in\N}{\sup}||g_n||_{L^1(\Omega_{k,m})}<\infty$.
  Altogether, we obtain for arbitrary $k,m\in\N$ and $\varphi\in L^{\infty}(\Omega_{k,m})$ that
  \begin{align}
   \left|\int_{\Omega_{k,m}}\varphi (f_ng_n-fg)\,dx\right|\leq &\left|\int_{\Omega}\varphi(f_n-f)g_n\,dx\right|+\left|\int_{\Omega_{k,m}}\varphi f(g_n-g)\,dx\right|\nonumber\\
   \leq &||\varphi||_{L^{\infty}(\Omega_{k,m})}\underset{n\in\N}{\sup}||g_n||_{L^1(\Omega_{k,m})}||f_n-f||_{L^{\infty}(\Omega_{k,m})}+\left|\int_{\Omega_{k,m}}\varphi f(g_n-g)\,dx\right|\underset{n\rightarrow\infty}{\rightarrow}0.\nonumber
  \end{align}
It follows that $f_ng_n\underset{n\rightarrow\infty}{\rightharpoonup}fg$ in $L^1(\Omega_{k,m})$. On the other hand, $f_ng_n\underset{n\rightarrow\infty}{\rightharpoonup}\xi$ in $L^1(\Omega)$, and, hence, also in $L^1(\Omega_{k,m})$. Consequently, $fg=\xi$ a.e. in $\Omega_{k,m}$ for all $k,m\in\N$. But $|\Omega_k\backslash\Omega_{m}|\leq\frac{1}{m}\underset{n^{(1)}\rightarrow\infty}{\rightarrow}0$ and $|\Omega\backslash \underset{k\in\N}{\bigcup}\Omega_k|=0$, so that $fg=\xi$  holds a.e. in $\Omega$.
 \end{proof}
A similar result holds for sums of member-by-member products.
\begin{Lemma}[Weak-a.e. convergence for sums]\label{LemA2}
 Let $\Omega$ be a measurable subset of $\R^N$ with finite measure and let $L\in\N$. Let $f^l,f^l_n:\Omega\rightarrow\R$, $n\in\N$, $l\in\{1,...,L\}$, be measurable functions and $g^l,g^l_n\in L^1(\Omega)$, $n\in\N$, $l\in\{1,...,L\}$. Assume further that $f^l_n\underset{n\rightarrow\infty}{\rightarrow} f^l$ a.e. in $\Omega$ and $g^l_n\underset{n\rightarrow\infty}{\rightharpoonup}g^l$, $\sum_{l=1}^{L}f^l_ng^l_n\underset{n\rightarrow\infty}{\rightharpoonup}\xi$ in $L^1(\Omega)$. Then, it holds that $\xi=\sum_{l=1}^{L}f^lg^l$ a.e. in $\Omega$.
 \end{Lemma}
 \begin{Remark}
  Observe that, in {\it Lemma~\ref{LemA2}}, we require not the  sequences $\left\{f^l_ng^l_n\right\}_{n\in\N}$ themselves  to be convergent for $l\in\{1,...,L\}$, but only their sum $\left\{\sum_{l=1}^{L}f^l_ng^l_n\right\}_{n\in\N}$.  Thus, the result is applicable in the cases where the convergence of individual sequences is either false or unknown.
 \end{Remark}
 \noindent
The proof of {\it Lemma~\ref{LemA2}} is very similar to the proof of {\it Lemma~\ref{LemA1}}. One only has to choose the sets $\Omega_k$ and  $\Omega_{k,m}$  independent of $l\in\{1,...,L\}$. We leave the remaining details to the reader.
\end{appendices}

\end{document}